\documentclass[11pt]{article}
\usepackage[a4paper, total={6in,8in}]{geometry}
\usepackage{amsmath}
\usepackage{amsfonts}
\usepackage{amsthm}
\usepackage{graphicx}
\usepackage{caption}
\usepackage{subcaption}
\usepackage{setspace}
\usepackage{hyperref}
\usepackage{epstopdf}
\usepackage{cleveref} 
\usepackage{wrapfig}
\usepackage{float}
\usepackage{multicol}

\theoremstyle{definition}
\newtheorem{theorem}{Theorem}[section]
\newtheorem{remark}[theorem]{Remark}
\newtheorem{proposition}[theorem]{Proposition}
\newtheorem{definition}[theorem]{Definition}
\newtheorem{lemma}[theorem]{Lemma}
\newtheorem{corollary}[theorem]{Corollary}

\newtheorem{assumption}[theorem]{Assumption}

\providecommand{\keywords}[1]{\textbf{\textit{Keywords:}} #1}

\DeclareMathOperator*{\mn}{\mathcal{N}}
\DeclareMathOperator*{\mq}{\mathcal{Q}}
\DeclareMathOperator*{\mqc}{\overline{\mathcal{Q}}}
\DeclareMathOperator*{\mm}{\mathcal{M}}
\DeclareMathOperator*{\mt}{\mathbb{T}}
\DeclareMathOperator*{\Mt3}{{\mathbb{T}^3}}

\mathchardef\mhyphen="2D

\DeclareMathOperator*{\s2}{\mathbb{S}^2}
\DeclareMathOperator*{\r3}{\mathbb{R}^3}
\newcommand{\mf}{\mathcal{F}}
\newcommand{\mg}{\mathcal{G}}
\newcommand{\ws}{\overset{*}{\rightharpoonup}}

\newcommand{\h}[1]{\widehat{#1}}

\newcommand{\wl}{\rightharpoonup}

\DeclareMathOperator*{\sm3}{\text{Sym}_0(3)}
\DeclareMathOperator*{\so3}{\text{SO}(3)}

\DeclareMathOperator*{\supp}{supp}
\DeclareMathOperator*{\argmin}{arg \, min}

\begin{document}
\sloppy
\title{Oseen-Frank-type theories of ordered media as the $\Gamma$-limit of a non-local mean-field free energy}
\author{Jamie M. Taylor\footnote{Address for correspondence: Mathematical Institute, Radcliffe Observatory Quarter, Woodstock Road, Oxford, OX2 6GG. Email: \texttt{jamie.taylor@maths.ox.ac.uk}}}
\date{}
\maketitle
\begin{abstract}
In this work we recover the Oseen-Frank theory of nematic liquid crystals as a $\Gamma$-limit of a particular mean-field free energy as the sample size becomes infinitely large. The Frank constants are not necessarily all equal. Our analysis takes place in a broader framework however, also providing results for more general systems such as biaxial or polar molecules. We start from a mesoscopic model describing a competition between entropy and a non-local pairwise interactions between molecules. We assume the interaction potential is separable so that the energy can be reduced to a model involving a macroscopic order parameter. We assume only integrability and decay properties of the macroscopic interaction, but no regularity assumptions are required. In particular, singular interactions are permitted. The analysis becomes significantly simpler on a translationally invariant domain, so we first consider periodic domains with increasing domain of periodicity. Then we tackle the problem on a Lipschitz domain with non-local boundary conditions by considering an asymptotically equivalent problem on periodic domains. We conclude by applying the results to a case which reduces to the Oseen-Frank model of elasticity, and give expressions for the Frank constants in terms of integrals of the interaction kernel. 
\end{abstract}
\keywords{Liquid crystals, Mean-field free energy, Oseen-Frank, Gamma convergence}
\section{Introduction}

In the mean-field free energy we consider particles inhabiting some configurational space, which for the sake of this work will be taken as a location $x \in \Omega\subset\mathbb{R}^3$ and some orientation-like component $m \in \mn$, where $(\mn,\mu)$ is a finite measure space. For the motivating case of liquid crystals, particles will be rigid objects with some symmetry, so $\mn$ will be taken typically as $\text{SO}(3)$ for general molecules or $\s2$ for molecules with axial symmetry. In the case of a system at constant concentration, we can describe the system by probability distributions on the state space, that is some $f :\Omega\to \mathcal{P}(\mn)$. In this case $f$ describes the distribution of molecules at location $x$. We then consider minimisers of the mean-field free energy functional \cite{maier1959einfache}
\begin{equation}\label{eqMeanField}
\begin{split}
&\int_{\Omega\times\mn}f(x,m)\ln f(x,m)\,d(x,m)-\frac{1}{2}\int_{\Omega\times\mn}\int_{\Omega\times\mn}\mathcal{K}(x-y,m,m')f(x,m)f(y,m')\,d(x,m)\,d(y,m').
\end{split}
\end{equation}
The left-hand term represents the entropic part and favours a disordered system. The right-hand term represents pairwise molecular interactions, with $\mathcal{K}(z,m,m')$ corresponding to the interaction energy of two molecules oriented as $m,m'\in\mn$ respectively with centres of mass spanned by $z\in\r3$. The pairwise interactions will typically favour an ordered state, and it is the competition between these two terms leads to phase transitions mediated by temperature and concentration, although for the sake of this work these parameters have been absorbed into the interaction $\mathcal{K}$.

The Oseen-Frank like elastic theories we will be considering are models to describe similar systems at a much larger, macroscopic scale \cite{oseen1933theory}. They admit various interpretations, and the following discussion is relevant for the case considered in this work. We define a macroscopic order parameter $b$ as an average of some microscopic quantity. To be more specific, we assume a finite dimensional vector space $V$, and $a:\mn \to V$, and define for all $x$ in our domain 
\begin{equation}
b(x)=\int_{\mn}f(x,m)a(m)\,d\mu(m).
\end{equation}
Typically $b$ will live in some open subset of $V$. We assume that the energy of the system at the macroscopic level can be described as 
\begin{equation}\label{eqLdG}
\int_{\Omega}W(\nabla b(x))+\psi(b(x))\,dx,
\end{equation}
where $W$ is the elastic energy and $\psi$ is the bulk energy. By considering scaling arguments it can be argued that for a sufficiently large sample, the bulk energy overwhelms the elastic energy so we consider the constrained case of $b$ in the minimising manifold, $\mm =\argmin(\psi)$. Thus rather than considering $b$ in $V$ (or some open subset of it), we consider only the restricted class $b :\Omega\to\mm$ and minimise a purely elastic energy
\begin{equation}
\int_{\Omega}W(\nabla b(x))\,dx.
\end{equation}
These are the limiting models that we will be obtaining in this work, although they will be obtained from the non-local free energy \eqref{eqMeanField} rather than the local form \eqref{eqLdG}. The local and non-local energies (\ref{eqMeanField},\ref{eqLdG}) are not unrelated themselves however, and Taylor approximation arguments can provide a heuristic justification of the local models from the mean-field energy \cite{han2013microscopic}. In the case of uniaxial nematic liquid crystals, $\mm \sim\mathbb{R}P^2$, or in simpler cases $\mm$ is taken as $\s2$. In the case of biaxial nematics, $\mm$ is taken as $\text{SO}(3)$. 

The exact form of $W$ will vary depending on the system being considered, but the Landau expansion argument can provide a phenomenological basis. The undistorted state ($\nabla b=0$) is typically the ground state and $W$ is assumed to be sufficiently regular, so at least for configurations near the ground state we can replace $W$ with its second order Taylor approximation, which is constrained to respect the symmetry of the system. We will show in this work that certain quadratic energies for $b:\Omega\to\mm$ can be justified from the mean-field free energy on large domains via $\Gamma$-convergence, under assumptions on the interaction kernel $\mathcal{K}$. Explicitly, we consider a rescaled version of the energy in \eqref{eqMeanField} on a domain $\frac{1}{\epsilon}\Omega$ and take the asymptotics as $\epsilon \to 0$, obtaining the elastic energy for $b \in W^{1,2}(\Omega,\mm)$ as 
\begin{equation}
\frac{1}{4}\int_{\Omega}L\nabla b(x)\cdot \nabla b(x)\,dx
\end{equation} 
for an appropriate tensor $L$. The analysis will be presented in a very general framework, but we will conclude the paper by considering as a concrete case a set of interaction kernels that give rise to the Oseen-Frank elastic energy for uniaxial nematic liquid crystals with two distinct elastic constants. For analogous results on the convergence of solutions to minimisers of the local functional \eqref{eqLdG} in the case of nematic liquid crystals the reader is directed to \cite{majumdar2010landau}.

In the case where $\mn = \s2$ and $a(p)=p\otimes p-\frac{1}{3}I$, an asymptotic description of minimisers of \eqref{eqMeanField} is given by Liu and Wang \cite{liu2016oseen}, providing a first step towards a rigorous link between the mean-field and Oseen-Frank models. Within their work they consider a restricted class of interaction potentials, and show that critical points of \eqref{eqMeanField} converge to weakly harmonic maps in an appropriate sense. Weakly harmonic maps arise as solutions to the Euler-Lagrange equation for the minimisation of
\begin{equation}\label{eqHarmonic}
\int_\Omega |\nabla n(x)|^2\,dx
\end{equation}
over $n \in W^{1,2}(\Omega,\s2)$ subject to a Dirichlet boundary condition. Furthermore, they show that minimisers of \eqref{eqMeanField} converge to minimising harmonic maps, that is minimisers of \eqref{eqHarmonic} over its admissible class. 

Our work extends their ideas into several directions. Firstly, our convergence result is stated as a $\Gamma$-limit of the energy, providing information about the behaviour of the system out of equilibrium (see \cite{braides2002gamma} for a thorough reference on techniques of $\Gamma$-convergence). Secondly, we permit much more varied configuration spaces $\mm$ than just $\s2$, and more general kernels $\mathcal{K}$. This allows a description of more general molecules such as those with biaxial or polar character, and furthermore reclaims in a particular case the Oseen-Frank energy with multiple elastic constants, so that we obtain as our limiting energy 
\begin{equation}
\frac{1}{4}\int_{\Omega}K_1(\text{div }n)^2+K_2|n\times \text{curl }n|^2+K_3(n\cdot \text{curl }n)^2\,dx
\end{equation}
with $K_1=K_3\neq K_2$ for $n \in W^{1,2}(\Omega,\s2)$ with Dirichlet boundary conditions. We note that the harmonic map problem corresponds to $K_1=K_2=K_3$, the so-called {\it one-constant approximation}. Finally, within our work we are able to drop almost any regularity assumption, with the translational part of a separable interaction potential only required to be measurable, non-negative definite and satisfy a certain polynomial decay. 

In the mean-field free energy we {\it a priori} only have that $b \in L^\infty(\Omega,\mqc)$. This lack of regularity means that we cannot impose Dirichlet boundary conditions in the sense of traces as if $b$ were a Sobolev function. To this end, we consider a ``thick" boundary condition specified on a set of positive volume, so that there is a subdomain $\Omega_\epsilon\subset \Omega$, so that $b$ is constrained to a fixed value on $\Omega\setminus \Omega_\epsilon$. Under the assumptions we consider (see \Cref{subsubsecAssumptionsGeneral}) this gives, for each $\epsilon>0$, a neighbourhood of $\partial\Omega$ in which $b$ has constrained values. In the limit as $\epsilon \to 0$, we will see this reduce to a standard Dirichlet condition.

In our analysis we also consider the effect of electrostatic interactions. This has a more detailed discussion in \Cref{subsubsecElectrostatic}. In the absence of external charges and at equilibrium, this adds a term to the energy given by 
\begin{equation}
-\frac{1}{2}\int_{\Omega} A(b(x))\nabla \phi(x) \cdot \nabla \phi(x)\,dx,
\end{equation}
where $A:\mqc\to\mathbb{R}^{3\times 3}$ represents the anisotropy of the media, and the electrostatic $\phi$ is required to satisfy $\nabla \cdot (A(b)\nabla \phi)=0$ and a Dirichlet boundary condition. As $\phi$ is uniquely defined by $b$, the boundary condition and the differential equation, we can view this term as a function of $b$ alone, which is in fact continuous so does not pose difficulties in our analysis.

\subsection{Summary of the paper}
In \Cref{subsecDefinition} we introduce some preliminary definitions and assumptions for the model. In particular, by assuming a separability of the interaction kernel and considering a restricted class of probability distributions, we start from the mean-field free energy \eqref{eqMeanField} and obtain a simpler free energy in the order parameter $b$ with equivalent minimisers. The order parameter $b$ is constrained to live in a bounded convex set $\mqc$, relating to its definition as an average. The energy for $b \in L^\infty(\Omega,\mqc)$ is given as 
\begin{equation}
\int_{\Omega}\psi_s(b(x))\,dx - \frac{1}{2}\int_{\Omega}\int_{\Omega}K(x-y)b(x)\cdot b(y)\,dx\,dy.
\end{equation}
Here $\psi_s$ is a macroscopic analogue of entropy, as defined in \cite{taylor2016maximum}. 

Next in \Cref{secPeriodicDomains} we consider an energy which formally corresponds to the case where $\Omega=\r3$, and the system is $\frac{2\pi}{\epsilon}$-periodic in the three coordinate directions. We denote the 3-torus as $\mathbb{T}^3=[0,2\pi]^3$ with opposite faces identified. Since the integrals are unbounded, we consider the energy per periodic cell $\frac{1}{\epsilon}\mt^3$, 
\begin{equation}\label{eqEnergyPrime}
\int_{\frac{1}{\epsilon}\mt^3}\psi_s(b(x))\,dx -\frac{1}{2}\int_{\frac{1}{\epsilon}\mt^3}\int_{\r3}K(x-y)b(x)\cdot b(y)\,dy\,dx. 
\end{equation}
While this is a physically unrealistic situation, the mathematical analysis admits a simpler treatment than the case of bounded domains, and will be essential for considering general bounded domains. Under a rescaling, the energy for $\epsilon>0$ that we consider is 
\begin{equation}\label{eqPeriodicEnergy}
\mf_\epsilon(b)=\frac{1}{\epsilon^2}\int_{\mt^3}\psi(b(x))\,dx +\frac{1}{2\epsilon^5}\int_{\mt^3}\int_{\r3}K\left(\frac{x-y}{\epsilon}\right)\cdot(b(x)-b(y))^{\otimes 2}\,dy\,dx.
\end{equation}
Here $\psi:\mqc\to\mathbb{R}$ is an appropriate bulk potential. We use the notation here and throughout the paper for a symmetric operator $A$ and vector $v$ that $A\cdot (v)^{\otimes 2}=Av\cdot v$, denoting the inner product of tensors between $A$ and $v\otimes v=v^{\otimes 2}$. Up to additive and multiplicative constant, the expressions \eqref{eqEnergyPrime} and \eqref{eqPeriodicEnergy} are equal, however the form in \eqref{eqPeriodicEnergy} provides a more convenient form for our analysis. Loosely speaking it writes the energy as a bulk term weighted by $\frac{1}{\epsilon^2}$, and a finite difference quotient that will, in the limit, become a gradient. 

In \Cref{subsecEstimates} we provide some of the necessary estimates and compactness results for our analysis. We do not need to impose any kind of regularity property on $K$ for our analysis, and our assumptions permit singularities, provided they are integrable. The key point is to find appropriate smooth, periodic functions $\varphi_\epsilon$ so that $\varphi_\epsilon \to \delta$ in $\mathcal{D}'(\mt^3)$, so that we can bound 
\begin{equation}
\int_{\Omega}|\nabla (\varphi_\epsilon * b)(x)|^2\,dx\leq C \frac{1}{\epsilon^5}\int_{\mt^3}\int_{\r3}K\left(\frac{x-y}{\epsilon}\right)\cdot(b(x)-b(y))^{\otimes 2}\,dy\,dx
\end{equation}
as seen in \Cref{propW12Bound}. From this we obtain weak-$W^{1,2}$ compactness of $\varphi_\epsilon*b_\epsilon$ if $\mf_\epsilon (b_\epsilon)$ is bounded. We then show that $\varphi_\epsilon*b_\epsilon - b_\epsilon \to 0$ in $L^2$ if $\mf_\epsilon(b_\epsilon)$ is bounded. Furthermore the weighting on the bulk term implies that if the energy is bounded, then $b_\epsilon \to \psi^{-1}(0)=\mm$, which gives us our compactness theorem for the energy:

\begin{theorem}[Compactness]\label{theoremPeriodicCompactness}
Let $b_{\epsilon} \in L^\infty({\mt^3},\mqc)$ be such that $\mf_\epsilon(b_\epsilon)$ is uniformly bounded. Then there exists some $b \in W^{1,2}({\mt^3},\mm)$ and a subsequence $\epsilon_j$ so that $b_{\epsilon_j}\overset{L^2}{\to} b$.
\end{theorem}

Once we have our compactness result, the $\Gamma$-convergence result is more straightforward. We can re-write the bilinear term in the energy as a difference quotient, so that if ${D_{-\epsilon z}b=\frac{1}{\epsilon|z|}(b(x)-b(x-\epsilon z))}$,
\begin{equation}
\begin{split}
&\frac{1}{\epsilon^5}\int_{\mt^3}\int_{\mt^3}K\left(\frac{x-y}{\epsilon}\right)\cdot \left(b(x)-b(y)\right)^{\otimes 2}\,dy\,dx\\
=&\int_{\mt^3}\int_{\r3}|z|^2K(z)\cdot (D_{-\epsilon z}b(x))^{\otimes 2}\,dz\,dx.
\end{split}
\end{equation}
Using the well behaved properties of difference quotients on $W^{1,2}(\mt^3,\mathbb{R}^k)$, we show that if $b \in W^{1,2}(\mt^3,V)$, then the limit of this as $ \epsilon \to 0$ is given by 
\begin{equation}
\begin{split}
&\int_{\mt^3}\int_{\r3}K(z)\cdot (z\cdot \nabla b(x))^{\otimes 2}\,dz\,dx\\
=& \int_{\mt^3}L\nabla b(x)\cdot \nabla b(x)\,dx
\end{split}
\end{equation}
for an appropriate tensor $L$ (\Cref{propLimsup}). This fact allows us to use constant recovery sequences to show the limsup inequality for our $\Gamma$-limit. 

For the liminf inequality, the key point is to observe that for $g(z)=\lambda_{\min}(K(z))$, that our method of convergence and energy bounds imply that the map
\begin{equation}
(z,x)\mapsto |z|g(z)^\frac{1}{2}D_{-{\epsilon_j}z}b_{\epsilon_j}(x)
\end{equation} 
converges weakly in $L^2(\r3\times\mt^3)$ to $(z,x)\mapsto -g(z)^\frac{1}{2}z\cdot \nabla b(x)$ as $\epsilon_j \to 0$. From this point, the liminf inequality becomes a straightforward application of standard lower semicontinuity theorems and is given in \Cref{propLiminf}. Combining these results gives us our next theorem:

\begin{theorem}[$\Gamma$-convergence with periodic domains]
We have that $\mf_\epsilon \overset{\Gamma}{\to} \mf$, where 
\begin{equation}
\mf(b)=\frac{1}{4}\int_{\mt^3}L\nabla b(x)\cdot \nabla b(x)\,dx
\end{equation}
if $b \in W^{1,2}(\mt^3,\mm)$, and $+\infty$ otherwise. The method of convergence is $L^2$-strong convergence. 
\end{theorem}

In \Cref{secBoundedDomains} we now turn to the case where $\Omega$ is a bounded domain in $\mathbb{R}^3$. As mentioned previously we take a ``thick" boundary condition so that for $\epsilon >0$, $b$ is only free to vary on a subdomain $\Omega_\epsilon\subset \Omega$, with $\Omega_\epsilon$ sufficiently separated from $\partial\Omega$. We also require that $\Omega_\epsilon \to\Omega$ in a precise way. The difficulty that would arise on a bounded domain is that convolutions are generally poorly behaved at the boundary of a domain, even if convolving against a smooth function. To this end, we extend any admissible $b \in L^\infty(\Omega,\mqc)$ to some $b\in W^{1,2}_0([0,2\pi]^3,\mathbb{R}^k)$, abusing notation to identify $b$ with its extension. Since $W^{1,2}_0([0,2\pi]^3,\mathbb{R}^k)$ functions can be identified with functions in $W^{1,2}(\mt^3,\mathbb{R}^k)$, we are now able to embed our problem into the periodic case, subject to a constraint corresponding to the boundary condition. The bulk of \Cref{subsecGammaBoundedDomain} is dedicated to showing that, by identifying each admissible $b$ with its extension in $W^{1,2}(\mt^3,\mathbb{R}^k)$,
\begin{equation}
\begin{split}
&\frac{1}{\epsilon}\int_\Omega\psi_s(b(x))\,dx-\frac{1}{2\epsilon^5}\int_{\Omega}\int_{\Omega}K\left(\frac{x-y}{\epsilon}\right)b(x)\cdot b(y)\\
=&\frac{1}{\epsilon^2}\int_{\Omega_\epsilon}\psi(b(x))-c_5\,dx+\int_{\mt^3}\int_{\mt^3}K_\epsilon(x-y)(b(x)-b(y))\cdot (b(x)-b(y))\,dx\,dy+R_\epsilon(b)+m_\epsilon
\end{split}
\end{equation}
for some constants $m_\epsilon$ and a function $R_\epsilon(b)$ which tends to zero uniformly in $b$ (\Cref{theoremEquivalentEnergies}). This means we can instead consider an asymptotically equivalent energy in a simpler, periodic domain. We also include electrostatic interactions which provide a continuous, although non-local, term depending on $b$. In this case, we have the energy to be minimised $\mg_\epsilon$ in $b$ and the electrostatic potential $\phi$, uniquely determined by $b$ and Maxwell's equations as 
\begin{equation}
\begin{split}
\mg_\epsilon(b,\phi)=&\frac{1}{\epsilon^2}\int_{\Omega_\epsilon}\psi(b(x))\,dx+\frac{1}{2\epsilon^2}\int_{\mt^3}\int_{\mt^3}K_\epsilon(x-y)\cdot (b(x)-b(y))^{\otimes 2}\,dx\,dy\\
&-\frac{1}{2}\int_{\Omega}A(b(x))\nabla \phi(x)\cdot \nabla \phi(x)\,dx,
\end{split}
\end{equation}
where we have the constraint on $\phi$ that 
\begin{equation}
\nabla \cdot (A(b)\nabla \phi)=0,
\end{equation}
plus Dirichlet boundary conditions. We can write this as 
\begin{equation}
\mg_\epsilon(b,\phi)=\frac{1}{\epsilon^2}\int_{\Omega_\epsilon}\psi(b(x))\,dx+\frac{1}{2\epsilon^2}\int_{\mt^3}\int_{\mt^3}K_\epsilon(x-y)\cdot (b(x)-b(y))^{\otimes 2}\,dx\,dy+\mathcal{E}(b),
\end{equation}
where $\mathcal{E}(b)$ is continuous with respect to $L^2$ convergence.

We then apply our estimates and lower semicontinuity results from the periodic case to obtain the $\Gamma$-convergence result for bounded domains,
\begin{theorem}
The functionals $\mathcal{G}_\epsilon \overset{\Gamma}{\to}\mathcal{G}$, where 
\begin{equation}
\mathcal{G}(b)=\int_{\mt^3}\frac{1}{4}L\nabla b(x)\cdot \nabla b(x)\,dx +\frac{1}{2}\int_{\Omega}A(b)^{-1}D(x)\cdot D(x)\,dx +\mathcal{E}(b)
\end{equation}
if $b \in W^{1,2}(\mt^3,\mathbb{R}^k)$ with $b(x) \in \mm$ for almost every $x \in \Omega$ and $b=b_0$ on $\mt^3\setminus\Omega$, and is $+\infty$ otherwise. The mode of convergence is $L^2$ strong convergence. Furthermore let $b_\epsilon \overset{L^2}{\to} b$, and denote the solutions of the maximisation problem defining $\mathcal{E}(b_\epsilon)$ as $\Phi_\epsilon$. Then $\text{div}(A(b_\epsilon)\Phi_\epsilon)=0$ and $\Phi_\epsilon|_{\partial\Omega}=\phi_0$. Then $\Phi_\epsilon\overset{W^{1,2}}{\to}\Phi$, where $\text{div }(A(b)\nabla \Phi)=0$  and $\Phi|_{\partial\Omega}=\phi_0$.  
\end{theorem}

We conclude the paper by considering the case where $\mn=\mathbb{S}^2$, $V=\text{Sym}_0(3)$ is the space of traceless symmetric matrices, and $a(p)=p\otimes p-\frac{1}{3}I$. In this case the order parameter is often denoted $Q$ and referred to as the Q-tensor. By enforcing symmetry constraints on the kernel $K$, we see a very small class of bilinear forms are permissible, and obtain in these cases the classical Frank elastic constants. In this case we have that $\mm = \left\{s^*\left(n \otimes n-\frac{1}{3}I\right):n\in\s2 \right\}$, and in the case when $Q\in W^{1,2}(\Omega,\mm)$ can be written as $Q(x)=s^*\left(n(x)\otimes n(x)-\frac{1}{3}I\right)$ for $n \in W^{1,2}(\Omega,\s2)$, the elastic component of the energy reduces to 
\begin{equation}
\frac{1}{4}\int_{\Omega} K_1 (\nabla \cdot n)^2 + K_2 |n\cdot \nabla \times n|^2 +K_3 (n\cdot \nabla \times n)^2,
\end{equation}
with $K_1=K_3\neq K_2$. That is to say that the so-called {\it one-constant approximation} $K_1=K_2=K_3$ does not hold. The exact relations between $K_i$ and the kernel $K$ are given in \Cref{propFrankConstants}, with some numerically found values for a case inspired by the London dispersion forces relation given in \Cref{remarkNumericalConstants}. 

\subsection{Definition and assumptions}\label{subsecDefinition}

\begin{assumption}[Separability of the interaction kernel]
We assume there exists a finite dimensional real Hilbert space $V$, functions $a \in L^\infty(\mn,V)$ and $K:\r3\to B(V,V)$ so that the interaction kernel $\mathcal{K}$ can be separated as 
\begin{equation}
\mathcal{K}(x-y,m,m')=K(x-y)a(m)\cdot a(m').
\end{equation}
Without loss of generality we will often assume that $V=\mathbb{R}^k$ for some $k$. 
\end{assumption}

This is a simplifying assumption and will allow us to reduce our problem to that of a finite dimensional order parameter. Throughout the work we will need bounds on terms involving $K$, so we define 
\begin{equation}
g(z)=\lambda_{\min}(K(z)). 
\end{equation}

\begin{assumption}
We assume the following technical assumptions on $a,K,g$. 
\begin{enumerate}
\item \label{techAssumptionNonneg} $g$ is non-negative everywhere and bounded away from zero on some open set.
\item \label{techAssumptionIntegrable} $g \in L^1(\r3)$ and has finite second moment, so $\int_{\r3}g(z)|z|^2\,dz<+\infty$. 
\item \label{technAssumptionUpperBound} There exists a constant $M>0$ so that $\lambda_{\max}(K)\leq Mg$. 
\item \label{techAssumptionSymmetry} $K$ is a measurable, even function, so $K(z)=K(-z)$ for all $z \in \r3$. 
\item \label{techAssumptionPseudoHaar} $a$ satisfies the property that for all $(c,\xi) \in \mathbb{R}\times V\setminus\{(0,0)\}$, $\mu(\{m \in \mn : c+\xi \cdot a(m)=0\})=0$. 
\end{enumerate}
\end{assumption}

Assumptions \ref{techAssumptionNonneg},\ref{techAssumptionIntegrable} and \ref{technAssumptionUpperBound} will be required for the various coercivity and compactness estimates on the energy. In particular, \ref{technAssumptionUpperBound} means that $K$ can be estimated above and below by the same scalar function. 

Assumption \ref{techAssumptionSymmetry} with the estimates gives that $K \in L^1(\r3,B(V,V))$, and the symmetry assumption is required for several proofs. In the case of uniaxial nematics with head-to-tail symmetry, the evenness of $K$ corresponds to mirror symmetry of interactions

Finally \ref{techAssumptionPseudoHaar}, known as the {\it pseudo-Haar} condition, allows a more elegant analysis of the relationship between the microscopic and macroscopic problem. It is essentially a strong form of linear independence. In the common cases where $\mn$ is a connected analytic manifold and each component of $a$, denoted $a_i$, is an analytic function, it is equivalent to the set of functions $\{1,a_1,...,a_k\}$ being linearly independent \cite[Proposition 3]{taylor2016maximum}. We then define the following objects
\begin{definition}
Let $\mq\subset V$ be the set of admissible moments defined by 
\begin{equation}
\mq = \left\{ \int_{\mn}a(m)f(m)\,dm : f \in \mathcal{P}(\mn)\right\}.
\end{equation}
For $b \in \mq$, let $U(b)=\left\{f \in \mathcal{P}(\mn):\int_{\mn}a(m)f(m)\,dm=b\right\}$. Then we define the singular potential as 
\begin{equation}
\psi_s(b)=\min\limits_{f \in U(b)}\int_{\mn}f(m)\ln f(m)\,dm.
\end{equation}
\end{definition}

We recall from \cite{taylor2016maximum} the following results on $\mq,\psi_s$, which strongly require the pseudo-Haar condition. 
\begin{proposition}
\begin{enumerate}
\item $\mathcal{Q}$ is an open, bounded, non-empty convex set.
\item $\psi_s:\mathcal{Q}\to\mathbb{R}$ is strictly convex, $C^\infty$, and satisfies $\lim\limits_{b \to \partial\mathcal{Q}}\psi_s(b)=+\infty$. 
\item The minimisation problem defining $\psi_s$ admits a unique solution denoted $\rho_b \in \mathcal{P}(\mn)$, given by 
\begin{equation}
\rho_b(m)=\frac{1}{Z}\exp(\Lambda_b\cdot a(m))
\end{equation}
for some $C^\infty$ bijection $\Lambda_b :\mq \to V$. 
\item For every Lipschitz function $F:\mq\to\mathbb{R}$, the function $b\mapsto \psi_s(b)+F(b)$ admits a minimum. 
\end{enumerate}
\end{proposition}
In light of the blow up of $\psi_s$ at the boundary of $\mathcal{Q}$, with abuse of notation we identify it with its extension to $\mqc$ given by $\psi_s(b)=+\infty$ for $b \in \partial\mq$. 

Formally, given $f \in L^1(\Omega\times\mn,[0,+\infty])$, satisfying $\int_{\mn}f(x,m)\,dm=1$ and $\int_{\mn}f(x,m)\,dm=b(x)$ for $x\in\Omega$, we have that 
\begin{equation}
\begin{split}
&\int_{\Omega\times\mn}f(x,m)\ln f(x,m)\,d(x,m)\\
&-\frac{1}{2}\int_{\Omega\times\mn}\int_{\Omega\times\mn}K(x-y)a(m)\cdot a(m')f(x,m)f(y,m')\,d(x,m)\,d(y,m')\\
=&\int_{\Omega\times\mn}f(x,m)\ln f(x,m)\,d(x,m)-\frac{1}{2}\int_{\Omega}\int_{\Omega}K(x-y)b(x)\cdot b(y)\,dx\,dy\\
\geq & \int_{\Omega}\psi_s(b(x))\,dx - \frac{1}{2}\int_{\Omega}\int_{\Omega}K(x-y)b(x)\cdot b(y)\,dx\,dy,
\end{split}
\end{equation}
with equality if and only if $f(x,m)=\rho_{b(x)}(m)$ for almost every $x,m$. This allows the problem to be reduced entirely to one phrased in $L^\infty(\Omega,\mqc)$. The exact argument requires technical care, since for an $L^1(\Omega\times\mn)$ function it is not immediately clear how to interpret $\int_{\mn}f(x,m)a(m)\,dm$ or $\int_{\mn}f(x,m)\,dm$. In \cite{liu2016oseen} this was rigorously interpreted via duality in the case when $\mn=\s2$ and $a(p)=p\otimes p-\frac{1}{3}I$, so that $\int_{\Omega\times\mn} \big(f(x,m)a(m)-b(x)\big)\varphi(x)\,dx=0$ for all $\varphi \in \mathcal{D}(\Omega)$. Their argument extends in a straightforward manner to the case we are considering, so the details will be omitted here and the energy 
\begin{equation}
\int_{\Omega}\psi_s(b(x))\,dx - \frac{1}{2}\int_{\Omega}\int_{\Omega}K(x-y)b(x)\cdot b(y)\,dx\,dy
\end{equation}
will be taken as a black box. In fact the precise form of $\psi_s$ becomes irrelevant to much of the analysis in this work, with only a lower bound and lower semicontinuity on $\mqc$ being important.

We use $\mt^3$ to denote the 3-torus, often identified with $[0,2\pi]^3$. We note however that $W^{1,2}([0,2\pi]^3,\mathbb{R})\neq W^{1,2}(\mt^3,\mathbb{R})$, although $L^p([0,2\pi]^3,\mathbb{R})=L^p(\mt^3,\mathbb{R})$. Again with abuse of notation, we identify functions in $L^1(\r3)$ which are $2\pi$-periodic in the coordinate directions with $L^1(\mt^3)$. For $h \in L^1(\mt^3)$, the $L^1$ norm is only defined as the integral over a single domain of periodicity, i.e. 
\begin{equation}
||h||_{L^1(\mt^3)}=||h||_1=\int_{\mt^3}|h(x)|\,dx.
\end{equation}

We now define notation for simplifying integrals.

\begin{definition}\label{defPeriodicVersion}
Let $h\in L^1(\r3)$. For $\epsilon >0$, define 
\begin{equation}
h_\epsilon(z)=\frac{1}{\epsilon^3}\sum\limits_{k \in \mathbb{Z}^3}h\left(\frac{z+2\pi k}{\epsilon}\right).
\end{equation}
\end{definition}
The following results are readily verified, with only sketch proofs provided. 
\begin{proposition}
Let $h \in L^1(\r3)$. Then
\begin{enumerate}
\item If $h \geq 0$ then $||h_\epsilon||_{L^1(\mt^3)}=||h||_{L^1(\r3)}$. 
\item Even if $h$ is negative somewhere, $h \in L^1(\mt^3)$.
\item If $u \in L^\infty(\mt^3)$, then 
\begin{equation}
\frac{1}{\epsilon^3}\int_{\r3}h\left(\frac{x}{\epsilon}\right)u(x)\,dx = \int_{\mt^3}h_\epsilon(x)u(x)\,dx.
\end{equation}
\end{enumerate}
\end{proposition}
\begin{proof}
To prove result 1, first we must verify $h_\epsilon$ is $2\pi$-periodic in the coordinate directions, but this follows simply by translating $k$ in the definition of $h_\epsilon$. To show $h_\epsilon$ is integrable if $h\geq 0$, we see that 
\begin{equation}
\begin{split}\label{eqChangingIntegrals}
&\int_{\mt^3}|h_\epsilon(x)|\,dx\\
=&\frac{1}{\epsilon^3}\int_{\mt^3}\sum\limits_{k \in \mathbb{Z}^3}h\left(\frac{x+2\pi k}{\epsilon}\right)\,dx\\
=&\frac{1}{\epsilon^3}\int_{\bigcup\limits_{k \in \mathbb{Z}^3}(2\pi k +\mt^3)}h\left(\frac{x}{\epsilon}\right)\,dx\\
=& \frac{1}{\epsilon^3}\int_{\r3}h\left(\frac{x}{\epsilon}\right)\,dx\\
=& \int_{\r3}h(x)\,dx=||h||_1.
\end{split}
\end{equation}
To see result 2, we note that $|h_\epsilon|<|h|_\epsilon$, and the result follows by result 1. Result 3 follows by the same argument described in lines 2-4 of \eqref{eqChangingIntegrals} applied to $\int_{\mt^3}h_\epsilon(x)u(x)\,dx$.
\end{proof}

We also recall the definition of $\Gamma$-convergence \cite{braides2002gamma}. 

\begin{definition}
Let $X$ be a metric space and let $F_\epsilon,F:X\to[-\infty,\infty]$. We say that $F_\epsilon\overset{\Gamma}{\to}F$, with respect to the topology on $X$, if the following hold. 
\begin{enumerate}
\item (liminf inequality) For all $x_\epsilon \to x$, $\liminf\limits_{\epsilon \to 0} F_\epsilon(x_\epsilon)\geq F(x)$.
\item (limsup inequality) For every $x \in X$, there exists a sequence $x_\epsilon \to x$ so that $\lim\limits_{\epsilon \to 0} F_\epsilon(x_\epsilon)=F(x)$. 
\end{enumerate}
\end{definition}

\section{Periodic domains}\label{secPeriodicDomains}

First, rather than considering domains $\frac{1}{\epsilon}\Omega$, we take $\Omega=\r3$ and assume that $b$ is $\frac{2\pi}{\epsilon}$-periodic in the three coordinate directions. This is problematic in that the integrals will be infinite unless all energy vanishes, so instead we consider what is formally the energy per periodic cell. That is, 
\begin{equation}
\int_{\frac{1}{\epsilon}[0,2\pi]^3}\psi_s(\tilde b(\tilde x))\,d\tilde x-\frac{1}{2}\int_{\frac{1}{\epsilon}[0,2\pi]^3}\int_{\r3}K\left(x-y\right)\tilde b(\tilde x)\cdot \tilde b(\tilde y)\,d\tilde y\,d\tilde x.
\end{equation}
While such an arrangement may seem physically unrealistic, it will provide the mathematical framework to understand the more physically reasonable case to be addressed in \Cref{secBoundedDomains}. We first perform a change of variables, $x=\epsilon \tilde{x} \in [0,2\pi]^3$, $y=\epsilon\tilde{y}\in [0,2\pi]^3$, $b(x)=\tilde{b}(\epsilon x)$ so $b \in L^\infty(\mt^3,\mqc)$. Then the energy becomes 
\begin{equation}\label{eqChangeOfVariables}
\frac{1}{\epsilon^3}\int_{\mt^3}\psi_s(b(x))\,dx -\frac{1}{2\epsilon^6}\int_{\mt^3}\int_{\r3}K\left(\frac{x-y}{\epsilon}\right)b(x)\cdot b(y)\,dx\,dy. 
\end{equation}
As in \Cref{defPeriodicVersion} denoting $K_\epsilon(z)=\frac{1}{\epsilon^3}\sum\limits_{k \in \mathbb{Z}^3}K\left(\frac{z+2\pi k}{\epsilon}\right)$, we re-write the integral as 
\begin{equation}
\frac{1}{\epsilon^3}\int_{\mt^3}\psi_s(b(x))\,dx -\frac{1}{2\epsilon^3}\int_{\mt^3}\int_{\mt^3}K_\epsilon(x-y)b(x)\cdot b(y)\,dy\,dx. 
\end{equation}
Now we note that for vectors $u,v$ and a symmetric linear operator $A$, $2Au\cdot v = Au\cdot u +A v \cdot v - A(u-v)\cdot (u-v)$. Then we re-write the energy as 
\begin{equation}
\begin{split}
&\frac{1}{\epsilon^3}\int_{\mt^3}\psi_s(b(x))\,dx -\frac{1}{4\epsilon^3}\int_{\mt^3}\int_{\r3}K_\epsilon(x-y)b(x)\cdot b(x)+K_\epsilon(x-y)b(y)\cdot b(y)\,dx\,dy\\
&+\frac{1}{4\epsilon^3}\int_{\mt^3}\int_{\mt^3}K_\epsilon(x-y)\cdot(b(x)-b(y))^{\otimes 2}\,dx\,dy\\
=&\frac{1}{\epsilon^3}\int_{\mt^3}\psi_s(b(x))\,dx -\frac{1}{2\epsilon^3}\int_{\mt^3}\left(\int_{\mt^3}K_\epsilon(x-y)\,dy\right)b(x)\cdot b(x)\,dx\\
&+\frac{1}{4\epsilon^3}\int_{\mt^3}\int_{\mt^3}K_\epsilon(x-y)\left(b(x)-b(y)\right)^{\otimes 2}\,dx\,dy. 
\end{split}
\end{equation}
Now we note that $\int_{\mt^3}K_\epsilon(x-y)\,dy=\int_{\r3}K(z)\,dz=K_0$, a tensor independent of $x$. Thus if we define $c_5=\min\limits_{b \in \mathcal{Q}}\psi(b)-\frac{1}{2}K_0b\cdot b$, and $\psi(b)=\psi_s(b)-\frac{1}{2}Kb\cdot b -c_5$, then the energy, rescaled by an additive constant and dividing through by $\epsilon$, is
\begin{equation}
\mf_\epsilon(b)=\frac{1}{\epsilon^2}\int_{\mt^3}\psi(b(x))\,dx+\frac{1}{4\epsilon^2}\int_{\mt^3}\int_{\mt^3}K_\epsilon(x-y)\left(b(x)-b(y)\right)^{\otimes 2}\,dx\,dy.
\end{equation}
Furthermore $\min\limits_{b \in \mqc}\psi(b)=0$. It is this functional $\mf_\epsilon$ as written that we turn our attention to. 

\subsection{Estimates and compactness}\label{subsecEstimates}

We first provide estimates that will give us our required compactness results. The general idea is to show that if $b \in L^\infty(\mt^3,\mqc)$, then for an appropriate periodic mollifier $\varphi_\epsilon$, the $W^{1,2}$ norm of $\varphi_\epsilon * b$ can be estimated by an appropriate constant times $\mf_\epsilon(b)$. Furthermore, we will be able to estimate $||\varphi_\epsilon*b-b||^2_2=O(\epsilon^2\mf_\epsilon(b))$. This will give that if $\mf_\epsilon(b_\epsilon)$ is bounded, then $\varphi_\epsilon*b_\epsilon$ is admits a $W^{1,2}$ converging subsequence, and the sequence is asymptotically equivalent in $L^2$ to $b_\epsilon$, providing our compactness theorem. First we provide several preliminary estimates.

\begin{lemma}
Let $h \in L^1(\mt^3)$ be even and non-negative, and $b \in L^\infty(\mt^3,\mathbb{R}^k)$. Then 
\begin{equation}
\int_{\mt^3} \int_{\mt^3} (h*h)(x-y)|b(x)-b(y)|^2\,dx\,dy\leq 4||h||_1 \int_{\mt^3} \int_{\mt^3} h(x-y)|b(x)-b(y)|^2\,dx\,dy.
\end{equation}
\end{lemma}
\begin{proof}
First, we note that given any vector space, the triangle inequality gives us $|u-v|\leq |u-w|+|w-v|$. Squaring both sides gives $|u-v|^2\leq |u-w|^2+|w-v|^2+2|u-w||v-w|$, and applying Young's inequality we have $|u-v|^2\leq 2|u-w|^2+2|w-v|^2$. Thus
\begin{equation}
\begin{split}
&\int_{\mt^3}\int_{\mt^3}(h*h)(x-y)|b(x)-b(y)|^2\,dx\,dy\\
=& \int_{\mt^3}\int_{\mt^3}\int_{\mt^3} h(x-y-z)h(z)|b(x)-b(y)|^2\,dz\,dx\,dy\\
\leq & 2\int_{\mt^3}\int_{\mt^3}\int_{\mt^3} h(x-y-z)h(z) \left(|b(x)-b(x-z)|^2+|b(x-z)-b(y)|^2\right)\,dz\,dx\,dy\\
=& 2||h||_1\int_{\mt^3}\int_{\mt^3} h(z)|b(x)-b(x-z)|^2\,dz\,dx\\
&+ 2\int_{\mt^3}\int_{\mt^3}\int_{\mt^3} h(\xi-y)h(z)|b(\xi)-b(y)|^2\,dz\,d\xi\,dy\\
=&2||h||_1\int_{\mt^3}\int_{\mt^3} h(x-z)|b(x)-b(x-z)|^2\,dz\,dx+2||h||_1\int_{\mt^3}\int_{\mt^3} h(\xi-y)|b(\xi)-b(y)|^2\,dy\,d\xi\\
=& 4||h||_1 \int_{\mt^3} \int_{\mt^3} h(x-y)|b(x)-b(y)|^2\,dx\,dy.
\end{split}
\end{equation}
Note a change of variables, replacing $x$ by taking $\xi=x-z$ was used.
\end{proof}

In order provide estimates of $\varphi_\epsilon * b$, the mollifier need be somehow comparable to the convolution kernel $K$. The exact relation, and existence of such a mollifier is given in the next result. 

\begin{lemma}\label{lemmaMollifier}
Let $g \in L^1({\mt^3})$ be even, non-negative and non-zero. Then there exists some even periodic mollifier $\varphi_\epsilon \in C^\infty({\mt^3})$ and $c>0$ so that $\varphi_\epsilon \leq \frac{1}{c}g_\epsilon$ and $|\nabla \varphi_\epsilon|\leq \frac{1}{c\epsilon}g_\epsilon$. 
\end{lemma}
\begin{proof}
Let $\tilde{\varphi} \in C^\infty(\mathbb{R})$ be non-negative and satisfy $\tilde{\varphi}\leq g$ and $|\nabla \tilde{\varphi}|\leq g$. This can be done since we assume that there exists an open set on which $g$ is strictly bounded away from zero. Define $\varphi(x)=\frac{\tilde{\varphi}(x)+\tilde{\varphi}(-x)}{2||\tilde{\varphi}||_1}$. Then a straightforward algebraic exercise gives that $\varphi_\epsilon=\sum\limits_{k \in\mathbb{Z}^3}\frac{1}{\epsilon^3} \varphi\left(\frac{x+2\pi k}{\epsilon}\right)$ satisfies the required properties. 
\end{proof}

\begin{proposition}\label{propW12Bound}
Let $\varphi_\epsilon$ be as in \Cref{lemmaMollifier}. Then for all $b \in L^\infty({\mt^3},\mqc)$,
\begin{equation}
\int_{\mt^3} |\nabla \left(\varphi_\epsilon * b\right)|^2\leq \frac{2||g||_1}{c^2\epsilon^2}\left(\int_{\mt^3} \int_{\mt^3}K_\epsilon(x-y)\cdot (b(x)-b(y))^{\otimes 2}\,dy\,dx\right).
\end{equation}
\end{proposition}
\begin{proof}
This is predominantly an algebraic exercise. First note that since $\varphi_\epsilon$ is even, its derivative is odd. Define the inner-product convolution $\hat{*}$ for $u,v:{\mt^3} \to \mathbb{R}^n$, $u\hat{*}v :{\mt^3} \to \mathbb{R}$ by
\begin{equation}
u\hat{*}v (x)=\int_{\mt^3} u(x-y)\cdot v(y)\,dy. 
\end{equation} Then we proceed as  
\begin{equation}
\begin{split}
\int_{\mt^3} |\nabla (\varphi_\epsilon*b_\epsilon)|^2=&\int_{\mt^3}\int_{\mt^3}\int_{\mt^3}\bigg( \nabla \varphi_\epsilon (x-y)\cdot \nabla \varphi_\epsilon(x-z)\bigg)\bigg(b_\epsilon(y)\cdot b_\epsilon(z)\bigg)\,dy\,dz\,dx\\
=& \int_{\mt^3} \int_{\mt^3} b(y)\cdot b(z)\left(\int_{\mt^3} -\nabla \varphi_\epsilon(y-x)\cdot \nabla \varphi_\epsilon(x-z)\,dx\right)\,dy\,dx\\
=&  \int_{\mt^3} \int_{\mt^3} -b(z)\cdot b(y)\left(\nabla \varphi_\epsilon\hat{*}\nabla\varphi_\epsilon\right)(y-z)\,dy\,dz\\
=& \frac{1}{2}\int_{\mt^3} \int_{\mt^3} \left(|b(z)-b(y)|^2-|b(x)|^2-|b(y)|^2\right)\left(\nabla \varphi_\epsilon\hat{*}\nabla\varphi_\epsilon\right)(y-z)\\
\leq &\frac{1}{2}\int_{\mt^3} \int_{\mt^3} |b(z)-b(y)|^2|\nabla \varphi_\epsilon|*|\nabla \varphi_\epsilon|(y-z)\,dy\,dx\\
&-\int_{\mt^3} |b(z)|^2\,dz \int_{\mt^3} \left(\nabla \varphi_\epsilon\hat{*}\nabla\varphi_\epsilon\right)(y)\,dy\\
\leq & \frac{1}{2c^2\epsilon^2}\int_{\mt^3} \int_{\mt^3} |b(z)-b(y)|^2(g_\epsilon * g_\epsilon)(y-z)\,dy\,dz\\
\leq & \frac{2||g||_1}{c^2\epsilon^2}\int_{\mt^3} \int_{\mt^3} |b(z)-b(y)|^2 g_\epsilon(y-z)\,dy\,dz.
\end{split}
\end{equation}
Note that we used $\int_{\mt^3} \left(\nabla \varphi_\epsilon\hat{*}\nabla\varphi_\epsilon\right)(y)\,dy=\int_{\mt^3} \nabla\varphi_\epsilon(x)\,dx\cdot \int_{\mt^3} \nabla\varphi_\epsilon(y)\,dy=0$, since $\varphi_\epsilon$ is periodic. In this case, 
\begin{equation}
\begin{split}
\int_{\mt^3} |\nabla (\varphi_\epsilon*b_\epsilon)|^2\leq &\frac{2||g||_1}{c^2\epsilon^2}\int_{\mt^3} \int_{\mt^3} |b_\epsilon(z)-b_\epsilon(y)|^2g_\epsilon(y-z)\,dy\,dx\\
\leq & \frac{2||g||_1}{c^2\epsilon^2}\left(\int_{\mt^3} \int_{\mt^3}K_\epsilon(x-y)\cdot (b_\epsilon(x)-b_\epsilon(y))^{\otimes 2}\,dy\,dx\right).
\end{split}
\end{equation}
\end{proof}

\begin{corollary}\label{corollaryW12Limit}
Let $b_\epsilon \in L^\infty({\mt^3},\mathbb{R}^k)$ for $\epsilon>0$. Assume that 
\begin{equation}
\int_{\mt^3} \int_{\mt^3}K_\epsilon(x-y)\cdot (b_\epsilon(x)-b_\epsilon(y))^{\otimes 2}\,dy\,dx
\end{equation} is uniformly bounded. Then up to a subsequence, $b_{\epsilon_j} \to b$ in $L^2$ with $b \in W^{1,2}({\mt^3},\mathbb{R}^k)$. 
\end{corollary}
\begin{proof}
We can take a subsequence so that $b_{\epsilon_j}\ws b \in L^\infty({\mt^3},\mathbb{R}^k)$. Let $\varphi_\epsilon$ be as in \Cref{lemmaMollifier}. Then by the previous result we have that $\varphi_{\epsilon_j}* b_{\epsilon_j}$ is bounded in $W^{1,2}$, and therefore admits a weakly converging subsequence (not relabelled).  Furthermore, for any $h \in L^1$, $\langle \varphi_{\epsilon_j}*b_{\epsilon_j}-b_{\epsilon_j},h\rangle = \langle b_{\epsilon_j},\varphi_{\epsilon_j}* h -h\rangle$, which tends to zero since $\varphi_{\epsilon_j}*h \to h$ in $L^1$. Therefore the weak-* limit of $b_{\epsilon_j}$ and the weak limit of $\varphi_{\epsilon_j}* b_{\epsilon_j}$ coincide, with the former known to be in $W^{1,2}$. Now since $\varphi_\epsilon * b_{\epsilon_j} \wl b$ in $W^{1,2}$, we can take a subsequence (not relabelled) so that $\varphi_\epsilon*b_{\epsilon_j}\to b$ in $L^2$. Then finally we see that 
\begin{equation}
\begin{split}
&\int_{\mt^3} |\varphi_{\epsilon_j}*b_{\epsilon_j}(x)-b_{\epsilon_j}(x)|^2\,dx\\
=&\int_{\mt^3} \left|\int_{\mt^3}\varphi_{\epsilon_j}(x-y)(b_{\epsilon_j}(y)-b_{\epsilon_j}(x))\,dy\right|^2\,dx\\
=& \int_{\mt^3} \left|\int_{\mt^3}\varphi_{\epsilon_j}(x-y)^\frac{1}{2}\left(\varphi_{\epsilon_j}(x-y)^\frac{1}{2}(b_{\epsilon_j}(y)-b_{\epsilon_j}(x))\right)\,dy\right|^2\,dx\\
\leq & \int_{\mt^3} \int_{\mt^3} \varphi_{\epsilon_j}(x-y)|b_{\epsilon_j}(y)-b_{\epsilon_j}(x)|^2\,dx\,dy\\
\leq & \frac{1}{c}\int_{\mt^3} \int_{\mt^3} g_{\epsilon_j}(x-y)|b_{\epsilon_j}(x)-b_{\epsilon_j}(y)|^2\,dx\,dy\\
\leq & \frac{{\epsilon_j}^2}{c}\mf_{\epsilon_j}(b_{\epsilon_j})
\end{split}
\end{equation}
which, since $\mf_{\epsilon_j}(b_{\epsilon_j})$ is uniformly bounded, must then be of order ${\epsilon_j}^2$. Therefore $b_{\epsilon_j}$ admits the same $L^2$ limit as $\varphi_{\epsilon_j}*b_{\epsilon_j}$, which is known to be $b \in W^{1,2}({\mt^3},\mathbb{R}^k)$. 
\end{proof}

\begin{proposition}\label{propCompactMinimisingManifold}
If $\mf_{\epsilon_j}(b_{\epsilon_j})$ is uniformly bounded, then any cluster point $b$ as given in \Cref{corollaryW12Limit} must satisfy $b(x)\in \mm = \psi^{-1}(0)$ pointwise almost everywhere. 
\end{proposition}
\begin{proof}
Since $b_{\epsilon_j}\to b$ in $L^2$, we must have that 
\begin{equation}
\liminf\limits_{j \to \infty}\int_{\mt^3} \psi(b_{\epsilon_j}(x))\,dx \geq \int_{\mt^3} \psi(b(x))\,dx.
\end{equation}
The left-hand side is of order $\epsilon_j^2$, so therefore we have that $\int_{\mt^3} \psi(b(x))\,dx=0$, so that $\psi(b(x))=0$ almost everywhere. 
\end{proof}

Combining the previous results is our main compactness theorem. 

\begin{theorem}[Compactness]\label{theoremPeriodicCompactness}
Let $b_{\epsilon} \in L^\infty({\mt^3},\mqc)$ be such that $\mf_\epsilon(b_\epsilon)$ is uniformly bounded. Then there exists some $b \in W^{1,2}({\mt^3},\mm)$ and a subsequence $\epsilon_j$ so that $b_{\epsilon_j}\overset{L^2}{\to} b$.
\end{theorem}

\subsection{The $\Gamma$-limit}

Now that we have our compactness result, we turn to evaluating the $\Gamma$-limit itself. The key idea is to rewrite the bilinear form as a bilinear form involving finite difference quotients, which will in the limit become gradients. To this end, we must first recall some properties of translation and finite difference operators on function spaces. 

\begin{definition}
Let $z \in \mathbb{R}^3$, and $b \in L^\infty(\mt^3,\mqc)$. Then define the translation and finite difference operators, $T_z:L^\infty(\mt^3,\mqc)\to L^\infty(\mt^3,\mqc)$ and $D_z:L^\infty(\mt^3,\mqc)\to L^\infty(\mt^3,\mathbb{R}^k)$ respectively by 
\begin{equation}
\begin{split}
T_zb(x)=&b(x+z)\\
D_zb(x)=&\frac{1}{|z|}\left(b(x+z)-b\right)=\frac{1}{|z|}(T_z-I)b(x).
\end{split}
\end{equation}
\end{definition}

We now recall some elementary results of these operators, which in various forms can be found in standard references (e.g. \cite{gilbarg2015elliptic}), although proofs are included for completeness.
\begin{proposition}
Fix $u \in L^2(\mt^3,\mathbb{R})$. Then 
\begin{enumerate}
\item $T_zu \to u$ in $\mathcal{D}'(\mt^3)$.
\item The map from $\mt^3$ to $L^2(\mt^3,\mathbb{R})$ given by $z \to T_zu$ is continuous. 
\item If $u \in W^{1,2}(\mt^3,\mathbb{R})$, then $D_{tz}u \to \h{z}\cdot \nabla u$ in $\mathcal{D}'(\mt^3)$ as $t \to 0^+$. 
\item If $u \in W^{1,2}(\mt^3,\mathbb{R})$, then $||D_{tz}u||_2\leq ||\h{z}\cdot \nabla u||_2$ and  $||D_zu-\h{z}\cdot \nabla u ||_2\leq \sup\limits_{|\xi|\leq |z|}||(T_\xi-I)\nabla u||_2$, which tends to zero as $|z|\to 0$.
\end{enumerate}
\begin{proof}
\begin{enumerate}
\item Let $\phi \in \mathcal{D}(\mt^3)$. Then 
\begin{equation}
\begin{split}
\left|\langle T_zu,\phi\rangle-\langle u,\phi\rangle \right|= &\left| \int_{\mt^3}u(x+z)\phi(x)-u(x)\phi(x)\,dx\right|\\
=& \left|\int_{\mt^3}u(x)\big(\phi(x-z)-\phi(x)\big)\,dx\right|\\
\leq & \int_{\mt^3}|u(x)|\,|\phi(x-z)-\phi(x)|\,dx\\
\leq &||u||_2\left(\int_{\mt^3}\text{Lip }(\phi)^2|z|^2\,dx\right)^\frac{1}{2}\\
=&O(|z|). 
\end{split}
\end{equation}
\item It is immediate that $||T_zu||_2=||u||_2$ by translational invariance on $\mt^3$. Then since $T_zu \to u$ in $\mathcal{D}'(\mt^3)$ and $||T_zu||_2$ is bounded, this implies $T_zu \overset{L^2}{\wl} u$ in $L^2$. But also, $\lim\limits_{z \to 0}||T_zu||_2=\lim\limits_{z \to 0} ||u||_2 = ||u||_2$, so that $T_zu \overset{L^2}{\wl}u$ and $||T_zu||_2\to ||u||_2$, so $T_zu\overset{L^2}{\to} u$. 
\item Again let $\phi \in \mathcal{D}(\mt^3)$. We first note that $D_{ty}\phi(x) \to \h{y}\nabla \cdot \phi(x)$ uniformly in $x$ as $t \to 0^+$. 
\begin{equation}
\begin{split}
\langle D_zu,\phi\rangle=& \frac{1}{|z|}\int_{\mt^3}(u(x+z)-u(x))\phi(x)\,dx\\
=& \frac{1}{|z|}\int_{\mt^3}u(x)\left(\phi(x-z)-\phi(x)\right)\,dx\\
\to & \int_{\mt^3}u(x)\big(-\h{z}\cdot \nabla \phi(x)\big)\,dx\\
=& \int_{\mt^3}\h{z}\cdot\nabla  u(x) \phi(x)\,dx. 
\end{split}
\end{equation}
\item To show that $||D_zu||_2\leq ||\h{z}\cdot \nabla u||_2$, we see it suffices to only consider when $b \in C^1(\mt^3)$ and then the result follows by density. In this case we see that 
\begin{equation}
\begin{split}
|D_zu(x)|=& \frac{1}{|z|}\left|\int_{0}^{|z|}\nabla u(x+t\h{z})\cdot \h{z}\,dt\right|\\
\leq & \frac{1}{|z|}\left(\int_0^{|z|}\,dt\right)^\frac{1}{2}\left(\int_{0}^{|z|}|\nabla u(x+t\h{z})\cdot \h{z}|^2\,dt\right)^\frac{1}{2}
\\
\Rightarrow \int_{\mt^3}|D_zu(x)|^2\,dx \leq & \frac{1}{|z|}\int_{\mt^3}\int_{0}^{|z|}|\nabla u(x+t\h{z})\cdot \h{z}|^2\,dt\,dx=||\nabla u(x)\cdot \h{z}||_2^2. \end{split}
\end{equation}
To show the inequality, we again consider only $u \in C^1(\mt^3)$ and the result follows by density. 
\begin{equation}
\begin{split}
\int_{\mt^3}|D_zu(x)-\h{z}\cdot \nabla u(x)|^2 =& \int_{\mt^3}\left|\frac{1}{|z|}\int_0^{|z|}\h{z}\cdot \nabla u(x+t\h{z}) -\h{z}\cdot \nabla u(x)\,dt\right|^2\,dx\\
\leq & \frac{1}{|z|}\int_{\mt^3}\int_{0}^{|z|}|\nabla u(x+t\h{z})-\nabla u(x)|^2\,dt\,dx\\
=& \int_{0}^{|z|}||(T_{t\h{z}}-I)\nabla u||_2^2\\
\leq & \sup\limits_{|\xi|\leq |z|} ||(T_{\xi}-I)\nabla u||_2^2.
\end{split}
\end{equation}
That this tends to zero as $|z|\to 0$ follows from the fact that $\xi \to T_\xi\nabla u$ is continuous on a compact set and $T_0=I$. 
\end{enumerate}
\end{proof}
\end{proposition}

\begin{proposition}\label{propLimsup}
Let $b_\epsilon \in W^{1,2}(\mt^3,\mathbb{R}^k)$ with $b_j \overset{W^{1,2}}{\to}b$. Then 
\begin{equation}
\lim\limits_{\epsilon \to 0} \int_{\mt^3}\int_{\mt^3}\frac{1}{\epsilon^2}K_\epsilon(x-y)\cdot(b_\epsilon(x)-b_\epsilon(y))^{\otimes 2}\,dy \,dx=\int_{\mt^3} L\nabla b(x)\cdot \nabla b(x)\,dx,
\end{equation}
where 
\begin{equation}L\nabla b(x)\cdot \nabla b(x) = \sum\limits_{i,j=1}^3\left(\int_{\mathbb{R}^3}K(z)z_iz_j\,dz\right)\frac{\partial b}{\partial x_i}(x)\cdot \frac{\partial b}{\partial x_j}(x).
\end{equation}
\end{proposition}
\begin{proof}
We fully write out $L$, and recall the definition of $K_\epsilon$ as an integral over $\mathbb{R}^3$, and this gives 
\begin{equation}
\begin{split}
&\left|\int_{\mt^3}\int_{\mt^3}\frac{1}{\epsilon^2}K_\epsilon(x-y)\cdot(b_\epsilon(x)-b_\epsilon(y))^{\otimes 2}\,dy - L\nabla b(x)\cdot \nabla b(x)\,dx\right|\\
&\left|\int_{\mt^3}\int_{\r3}\frac{1}{\epsilon^5}K\left(\frac{x-y}{\epsilon}\right)\cdot(b_\epsilon(x)-b_\epsilon(y))^{\otimes 2}\,dy - \int_{\r3}K(z)(z\cdot\nabla b(x))\cdot (z\cdot\nabla b(x))\,dz\,dx\right|\\
=& \left|\int_{\mt^3}\int_{\mathbb{R}^3}\frac{1}{\epsilon^2}K(z)\cdot \left(b_\epsilon(x)-b_\epsilon(x-\epsilon z)\right)^{\otimes 2} - K(z)(z\cdot\nabla b(x))^{\otimes 2}\,dz\,dx\right|\\
\leq & \int_{\mt^3}\int_{\mathbb{R}^3}|K(z)| |z|^2\left|\left(D_{-\epsilon z}b_\epsilon(x)\right)^{\otimes 2}-(\h{z}\cdot \nabla b(x))^{\otimes 2}\right|\,dz\,dx\\
=& \int_{\mathbb{R}^3}|K(z)||z|^2\int_{\mt^3} \left|\left(D_{-\epsilon z}b_\epsilon(x)\right)^{\otimes 2}-(\h{z}\cdot \nabla b(x))^{\otimes 2}\right|\,dx\,dz\\
\leq & \int_{\mathbb{R}^3}|K(z)||z|^2\left(\int_{\mt^3}|D_{-\epsilon z}b_\epsilon(x)-\h{z}\cdot \nabla b(x)|^2\,dx\right)^\frac{1}{2}\left(\int_{\mt^3}|D_{-\epsilon z}b_\epsilon(x)+\h{z}\cdot \nabla b(x)|^2\,dx\right)^\frac{1}{2}\,dz\\
=&\int_{\mathbb{R}^3}|K(z)||z|^2||D_{-\epsilon z}b_\epsilon-\h{z}\cdot \nabla b||_2||D_{-\epsilon z}b_\epsilon+\h{z}\cdot \nabla b||_2\,dz \\
\leq &\int_{\mathbb{R}^3}|K(z)||z|^2\bigg(||D_{-\epsilon z}b_\epsilon||_2+||\nabla b||_2\bigg)\bigg(||D_{-\epsilon z}(b_\epsilon-b)||_2+||D_{-\epsilon z}b-\h{z}\cdot \nabla b||_2\bigg)\,dz \\
\leq &(||\nabla b_\epsilon||_2+||\nabla b||_2) \int_{\mathbb{R}^3}|K(z)||z|^2\left(||\nabla b_\epsilon - \nabla b||_2 + \sup\limits_{t\leq \epsilon|z|}||(T_{-t\h{z}}-I)\nabla b||_2\right)\,dz.
\end{split}
\end{equation}
Now we note that since $\nabla b_\epsilon\overset{L^2}{\to}\nabla b$ and $||(T_{-t\h{z}}-I)\nabla b||_2$ is bounded uniformly in $t,\h{z}$, the integrand is bounded by a constant times $|K(z)||z|^2$ which is integrable, so we may apply dominated convergence. Then using that $\sup\limits_{t\leq \epsilon|z|}||(T_{-t\h{z}}-I)\nabla b||_2\to 0$ pointwise as $\epsilon \to 0$ and the convergence of $b_\epsilon$, the result holds. 
\end{proof}

In order to show the liminf inequality, a key step is that, as a function on $\r3\times\mt^3$, $(z,x)\mapsto g(z)^\frac{1}{2}|z|D_{\epsilon z}Q_\epsilon(x)$ has the limit $g^\frac{1}{2}z\cdot \nabla Q(x)$ in the sense of distributions. The result is far more trivial if we integrate this against test functions which are separable as $\phi(z,x)=\phi_1(x)\phi_2(z)$. By means of the next lemma, we will show that the span of such test functions are dense in $\mathcal{D}(\r3\times\mt^3)$ with uniform convergence, allowing us to test only against such separable functions.

\begin{lemma}
Let $\mathcal{U}$ denote the subalgebra of $\mathcal{D}(\r3\times\mt^3)$ given by the span of separable functions, so that if $\phi \in \mathcal{U}$, $\phi(x,z)=\sum\limits_{i=1}^m \phi_1^i(x)\phi_2^i(z)$ for some $\phi_1^i \in \mathcal{D}(\mt^3)=C^\infty(\mt^3)$ and $\phi_2^i \in \mathcal{D}(\r3)$. Then $\mathcal{U}$ is dense with respect to uniform convergence in $\mathcal{D}(\r3\times\mt^3)$.
\end{lemma}
\begin{proof}
The result is essentially Stone-Weierstrass. Let $\phi \in \mathcal{D}(\r3\times\mt^3)$. Without loss of generality, assume $\supp(\phi)\subset B_1\times\mt^3$. Let $\mathcal{U}'$ denote the algebra given by the span of all smooth functions on $B_3\times\mt^3$ which are separable, but not necessarily with compact support. We note that this set separates points. This can be seen by taking classical bump functions with $\phi_1(x)=1$ if $x=x_0$ and $0\leq \phi_1(x)<1$ otherwise, and similarly $\phi_2(z)=1$ if $z=z_0$ and $\phi_2(z)\in[0,1)$ otherwise, then their product separates $(x_0,z_0)$ and every other point. Furthermore this is clearly an algebra, and contains a non-zero constant function. Therefore it is dense in $C(B_3\times\mt^3)$ by Stone-Weierstrass. Now let $\phi_0 \in C^\infty(\r3)$ be a cutoff function, so that $\phi_0(z)=1$ if $z \in B_1$, $\phi_0(z) \in (0,1)$ for $1<|z|<2$ and $\phi_0(z)$ if $|z|>2$. Then if $\phi_j \in \mathcal{U}'$ and $\phi_j \to \phi$ in $L^\infty$, then $\tilde{\phi}_j(x,z)=\phi_0(x)\phi_j(x,z)$ must also satisfy $\tilde{\phi}_j\to \phi$ in $L^\infty$, but also $\tilde \phi_j \in \mathcal{U}$. 
\end{proof}

\begin{proposition}\label{propLiminf}
If $L^\infty(\mt^3,\mathbb{R}^k)\ni b_{\epsilon_j} \overset{L^2}{\to} b \in W^{1,2}(\r3,\mathbb{R}^k)$, then 
\begin{equation}
\liminf\limits_{j \to \infty} \int_{\mt^3}\int_{\mt^3}\frac{1}{\epsilon^2}K_\epsilon(x-y)\cdot(b_{\epsilon_j}(x)-b_{\epsilon_j}(y))^{\otimes 2}\,dy \,dx\geq\int_{\mt^3} L\nabla b(x)\cdot \nabla b(x)\,dx.
\end{equation}
\end{proposition}
\begin{proof}
Assume the liminf is finite, else there is nothing to prove, and take a subsequence converging to liminf without relabelling. Since the liminf is finite, we can take a constant $C>0$ so that 
\begin{equation}
\begin{split}
C \geq & \int_{\mt^3}\int_{\mt^3}K_{\epsilon_j}(x-y)\left(b_{\epsilon_j}(x)-b_{\epsilon_j}(y)\right)^{\otimes 2}\,dx\,dy\\
=& \int_{\r3}\int_{\mt^3}K(z)\left(b_{\epsilon_j}(x)-b_{\epsilon_j}(x-{\epsilon_j} z)\right)^{\otimes 2}\,dx\,dz\\
=& \int_{\r3}\int_{\mt^3}|z|^2K(z)\cdot \left(D_{-{\epsilon_j}z}b_{\epsilon_j}(x)\right)^{\otimes 2}\,dx\,dz\\
\geq & \int_{\r3}\int_{\mt^3}|z|^2 g(z) |D_{-\epsilon_j z}b_{\epsilon_j}(x)|^2\,dx\,dz.
\end{split}
\end{equation}
This implies that $(z,x)\mapsto |z|g(z)^\frac{1}{2}D_{-{\epsilon_j}z}b_{\epsilon_j}(x)$ is bounded in $L^2(\r3\times\mt^3)$. Thus we take a weakly converging subsequence (not relabelled). We now seek to find the weak limit. We do this by finding the limit in the sense of distributions. In light of the previous lemma, it suffices to test against separable functions. 

Let $\phi \in \mathcal{D}(\r3\times\mt^3)$, with $\phi(x,z)=\phi_1(x)\phi_2(z)$. Then 
\begin{equation}
\begin{split}
&\left|\int_{\r3\times\mt^3}|z|g(z)^\frac{1}{2}D_{-{\epsilon_j}z}b_{\epsilon_j}(x)\phi(x,z)-g(z)^\frac{1}{2}z\cdot \nabla_x\phi(x,z)b(x)\,d(x,z)\right|\\
= & \left|\int_{\r3\times\mt^3}|z|g(z)^\frac{1}{2}\bigg (b_{\epsilon_j}(x)D_{{\epsilon_j}z}\phi(x,z)-b(x)\h{z}\cdot \nabla_x\phi(x,z)\bigg)\,d(x,z)\right|\\
\leq & \int_{\r3\times\mt^3}|z|g(z)^\frac{1}{2}\bigg(|b_{\epsilon_j}(x)|\,|\phi_2(z)|\,|D_{{\epsilon_j}z}\phi_1(x)-\h{z}\cdot\nabla\phi_1(x)|+|\nabla_x \phi(x,z)|\,|b_{{\epsilon_j}}(x)-b(x)|\bigg)\,d(x,z)\\
\leq & \left(\int_{\r3}|z|^2g(z)\,dz\int_{\mt^3}|b_{\epsilon_j}(x)|^2\,dx\right)^\frac{1}{2}\left(\int_{\r3}|\phi_2(z)|^2||D_{{\epsilon_j}z}\phi_1-\h{z}\cdot \nabla\phi_1||_2^2\,dz\right)^\frac{1}{2}\\
&+ ||\nabla_x \phi||_2 \left(\int_{\r3}|z|^2g(z)\,dz\right)^\frac{1}{2}||b_{\epsilon_j}-b||_2\\
\leq & C_1\left(\int_{\r3}|\phi_2(z)|^2\sup\limits_{|\xi|\leq {\epsilon_j}z}||(T_\xi-I)\nabla \phi_1||_2^2\,dz\right)^\frac{1}{2}+C_2||b_{\epsilon_j}-b||_2.
\end{split}
\end{equation}
To deal with the term in the integral, we recall that $\sup\limits_{|\xi|<\epsilon_j z}||(T_\xi-I)\nabla \phi||_2$ is bounded and converges to zero pointwise as $\epsilon_j \to 0$, so we can apply dominated convergence to give that the integral tends to zero in the limit. The second term tends to zero trivially, thus $|z|g(z)^\frac{1}{2}D_{-\epsilon_j z}b_{\epsilon_j}(x)\overset{L^2}{\wl}-g(z)^\frac{1}{2}\h{z}\cdot \nabla b(x)$. We take the convention that when $K(z)=0$, $\frac{1}{g}K(z)=0$ so that the proceeding equalities hold. Now we can apply standard lower semicontinuity results for convex functionals under weak convergence (see e.g. \cite{dacorogna2007direct}), to give that 
\begin{equation}
\begin{split}
&\liminf\limits_{j\to \infty}\frac{1}{{\epsilon_j}^2}\int_{\mt^3}\int_{\mt^3}K_\epsilon(x-y)\cdot (b_{\epsilon_j}(x)-b_{\epsilon_j}(y))^{\otimes 2}\,dy\,dx\\
=& \liminf\limits_{j \to \infty}\int_{\r3\times\mt^3}|z|^2K(z)\cdot (D_{-{\epsilon_j}z}b_{\epsilon_j}(x))^{\otimes 2}\,d(z,x)\\
=& \liminf\limits_{j \to \infty}\int_{\r3\times\mt^3}\left(\frac{1}{g(z)}K(z)\right)\cdot \left(|z|g(z)^\frac{1}{2}D_{-{\epsilon_j}z}b_{\epsilon_j}(x)\right)^{\otimes 2}\,d(z,x)\\
\geq &\int_{\r3\times\mt^3}\frac{1}{g(z)}K(z)\cdot \left(|z|g(z)^\frac{1}{2}\h{z}\cdot \nabla b(x)\right)^{\otimes 2}\,d(z,x)\\
=& \int_{\r3\times\mt^3}K(z)\cdot \left(z\cdot \nabla b(x)\right)^{\otimes 2}\,d(z,x)\\
=& \int_{\mt^3}L\nabla b(x)\cdot \nabla b(x)\,dx. 
\end{split}
\end{equation}
\end{proof}

\begin{theorem}[$\Gamma$-convergence with periodic domains]\label{theoremGammaPeriodic}
We have the $\Gamma$-convergence $\mf_\epsilon \overset{\Gamma}{\to} \mf$, where 
\begin{equation}
\mf(b)=\frac{1}{4}\int_{\mt^3}L\nabla b(x)\cdot \nabla b(x)\,dx
\end{equation}
if $b \in W^{1,2}(\mt^3,\mm)$, and $+\infty$ otherwise. The method of convergence is $L^2$-strong convergence. 
\end{theorem}
\begin{proof}
For the liminf inequality, since $\psi$ is non-negative, using \Cref{propLiminf},
\begin{equation}
\begin{split}
\liminf\limits_{j \to \infty} \mf_{\epsilon_j}(b_{\epsilon_j}) \geq &\liminf\limits_{\epsilon \to 0}\frac{1}{4\epsilon^2}\int_{\mt^3}\int_{\mt^3}K_\epsilon(x-y)\cdot(b_{\epsilon_j}(x)-b_{\epsilon_j}(y))^{\otimes 2}\,dx\,dy \\
\geq & \frac{1}{4}\int_{\mt^3}L\nabla b(x)\cdot \nabla b(x)\,dx.
\end{split}
\end{equation}
For the limsup inequality, we are fortunate in that we can take constant recovery sequences from \Cref{propLimsup}, so that if $b \in W^{1,2}(\mt^3,\mm)$, then $\psi(b)=0$ almost everywhere, and 
\begin{equation}
\begin{split}
\mf_{\epsilon_j}(b)=& \frac{1}{4\epsilon^2}\int_{\mt^3}\int_{\mt^3}K_\epsilon(x-y)\cdot(b(x)-b(y))^{\otimes 2}\,dx\,dy\\
\to & \frac{1}{4}\int_{\mt^3}L\nabla b(x)\cdot \nabla b(x)\,dx. 
\end{split}
\end{equation}
As per the compactness result, if $b \not\in W^{1,2}(\mt^3,\mm)$, then $\mf_{\epsilon_j}(b_{\epsilon_j})$ must blow up if $b_{\epsilon_j} \to b$. 
\end{proof}

\begin{remark}
Our systems will more often than not have to respect the principle of frame indifference. In order to interpret frame indifference we must be able to act rotations $R \in \so3$ on the order parameter space $V$. We denote this group action as $(R,b)\mapsto [R]b$. For example if $V=\mathbb{R}^3$, then $[R]b=Rb$, if $V=\sm3$ as in the Q-tensor theory, then $[R]b=RbR^T$. The kernel will need to satisfy frame indifference also, so that $[R]K(z)[R]^T=K(Rz)$, giving $K(Rz)[R]b\cdot [R]b=K(z)b\cdot b$. Given $b:\Omega\to\mqc$, the act of rotating the system by $R \in \so3$ implies that $\nabla b$ is replaced by $[R]\nabla b R^T$. Therefore frame indifference requires that $L\left([R]AR^T \right)\cdot\left( [R]AR^T\right)= LA\cdot A$ for all $A$ in the tangent space of $\mm$. Should we have frame indifference of $K$ this holds however, as seen in the next proposition.
\end{remark}

\begin{proposition}
Assume $\so3$ is defines group action on $V$ by $R\mapsto [R] \in B(V,V)$, and $K$ satisfies $K(Rz)=[R]K(z)[R]^T$. Then the quadratic form $A\mapsto LA\cdot A$ for $A$ in the tangent space of $\mm$ is frame invariant, so that $L([R]AR^T)\cdot ([R]AR^T) =LA\cdot A$.
\end{proposition}
\begin{proof}
\begin{equation}
\begin{split}
&L[R]AR^T\cdot [R]AR^T\\
=&\int_{\r3}K(z)\big([R]AR^Tz\big)\cdot \big( [R]AR^Tz\big)\,dz\\
=& \int_{\r3}\big([R]^TK(z)[R]\big) A(R^Tz) \cdot A(R^Tz)\,dz\\
=& \int_{\r3}K(R^Tz) A(R^Tz) \cdot A(R^Tz)\,dz\\
=& \int_{\r3}K(y)Ay\cdot Ay\,dy=LA\cdot A,
\end{split}
\end{equation}
by performing a change of variables $y=R^Tz$. 
\end{proof}

\subsection{Local minimisers}

Since we consider periodic domains and limiting energies purely quadratic in the gradient, we have that global minimisers are in all cases trivial, i.e. any constant map. However, particularly due to the role of the topology of the ground state manifold $\mm$, we expect {\it local} minimisers to provide interesting behaviour, in particular since the energy is typically not convex. We must however do some work to say anything about local minimisers. Standard results of $\Gamma$-convergence tell us that isolated local minimisers of $\mf$ can be approximated by local minimisers of $\mf_\epsilon$ \cite[Theorem 4.1.1]{braides2012local}, but the translational symmetry means that the only isolated local minimiser is the undistorted state. We now aim to reproduce the result for our case with continuous symmetry. Heuristically, we are applying the same proof as the given reference to the equivalence classes by identifying translations of functions.

\begin{definition}
Let $F:L^\infty(\mt^3,\mqc)\to\mathbb{R}\cup\{+\infty\}$. We say that $b_0 \in L^\infty({\mt^3},\mqc)$ is an isolated $L^2$-local minimiser of $F$, modulo translations, if there exists $\delta>0$ so that for all $b \in L^\infty({\mt^3},\mqc)$, if $||b-b_0||_2<\delta$ and $F(b)\leq F(b_0)$, then $b=T_zb_0$ for some $z \in \mathbb{R}^3$.
\end{definition}
\begin{remark}
Note that if $b$ is an isolated $L^2$-local minimiser modulo translations, or a (potentially non-strict) $L^2$-local minimiser, then so is $T_zb$ for all $z$, which follows immediately from the invariance of the energy and that $T_z$ is an automorphism of $L^2({\mt^3},\mqc)$. 
\end{remark}

\begin{lemma}
For $\epsilon>0$, and any closed ball $\bar{B}\subset L^2(\mt^3,\mqc)$ we have that $\min\limits_{b \in \bar{B}}\mf_\epsilon(b)$ exists.
\end{lemma}
\begin{proof}
We proceed with a standard direct method argument. We first recall that $\mf_\epsilon$ can be written as 
\begin{equation}
\frac{1}{\epsilon^2}\int_{\mt^3}\psi_s(b(x))\,dx-\frac{1}{2\epsilon^5}\int_{\mt^3}\int_{\mt^3}K_\epsilon(x-y)b(x)\cdot b(y)\,dy\,dx.
\end{equation}
$\mf_\epsilon$ trivially admits a lower bound on $\bar{B}$, and $\bar{B}$ is weakly compact in $L^2$, as a convex, strongly closed set bounded in $L^2$-norm. Thus it only remains to show lower semicontinuity. The convexity of $\psi_s$ gives lower semicontinuity of the entropic term. The bilinear form remains to be considered. If $(x,y)\mapsto K_\epsilon(x-y)$ is in $L^2(\mt^3\times\mt^3)$, then the convolution $b\mapsto \int_{\mt^3}K_\epsilon(\cdot-y)b(y)\,dy$ defines a compact linear operator from $L^2$ to $L^2$, which would prove the result. However our assumptions only guarantee that $(x,y)\mapsto K_\epsilon(x-y)$ is an $ L^1$ function. We use instead that if $b_{\epsilon_j}\overset{L^2}{\rightharpoonup}b$, then $b_{\epsilon_j}\overset{L^1}{\rightharpoonup}b$ also. In light of \cite[Corollary 4.1]{eveson1995compactness}, we see that $b\mapsto K_\epsilon * b$ is compact from $L^1$ to $L^\infty$ if and only if $\{T_{-y}K_\epsilon: y \in \mt^3\}$ is relatively compact in $L^1$. This however holds, since $\mt^3$ is compact, and $y \mapsto T_{-y}K_\epsilon$ is continuous in $L^1(\mt^3)$.
\end{proof}

\begin{proposition}
Let $b_0$ be an isolated $L^2$-local minimiser of $\mf$ modulo translations. Then there exists some sequence $\epsilon_j \to 0$ and $L^2$-local minimisers of $\mf_{\epsilon_j}$ converging to $b_0$ in $L^2$.
\end{proposition}
\begin{proof}
The main ingredients of the proof are that constant recovery sequences can be used, and that $T_z$ is an automorphism for all $z\in\mathbb{R}^3$.

Assume that $\delta>0$ is such that $||b_0-b||_2<2\delta$ and $\mf(b_0)\geq \mf(b)$ implies $b_0=T_zb$ for some $z$. Let $\bar{B}$ be the closed ball of radius $\delta$ about $b_0$. Now let $b_\epsilon \in \argmin\limits_{\bar{B}}\mf_\epsilon$ for each $\epsilon>0$, which must be a non-empty set. It must hold that $\mf_\epsilon(b_\epsilon) \leq \mf_\epsilon(b_0)$, and since $\mf_\epsilon(b_0)\to\mf(b_0)$ this implies $\mf_\epsilon(b_\epsilon)$ is bounded. We thus take a subsequence $b_{\epsilon_j}$ so that $b_{\epsilon_j}\overset{L^2}{\to} b^* \in \bar{B}$. Then we have 
\begin{equation}
\begin{split}
\mf(b^*)\leq &\liminf\limits_{j \to \infty} \mf_{\epsilon_j}(b_{\epsilon_j})\\
\leq & \liminf\limits_{j \to \infty} \mf_{\epsilon_j}(b_0)\\
=& \lim\limits_{\epsilon \to 0} \mf_\epsilon (b_0)\\
=& \mf(b). 
\end{split}
\end{equation}
Thus $b^*$ is so that $||b^*-b_0||_2\leq \delta <2\delta$ and $\mf(b^*)\leq \mf(b)$, so we must have that there exists some $z$ with $T_zb_0=b^*$. Since $T_z$ is continuous, this means that $T_{-z}b_{\epsilon_j} \to T_{-z}b^*=T_{-z}T_zb_0=b_0$. By invariance of the energy, $\mf_{\epsilon_j}(T_{-z}b_{\epsilon_j})=\mf_{\epsilon_j}(b_{\epsilon_j})=\min\limits_{\bar{B}}\mf_{\epsilon_j}$. Furthermore, for $j$ sufficiently large, since $T_{-z}b_{\epsilon_j}\to b_0$, this means $T_{-z}b_{\epsilon_j} \in \text{int }(B)$. This means that $T_{-z}b_{\epsilon_j}$ is a minimiser of $\mf_{\epsilon_j}$ on an open set in $L^2$, i.e. it is an $L^2$-local minimiser. Therefore $T_{-z}b_{\epsilon_j}$ is a sequence of $L^2$-local minimisers of $\mf_{\epsilon_j}$ converging to $b_0$ in $L^2$.
\end{proof}

\section{Bounded domains with electrostatic interactions}\label{secBoundedDomains}

\subsection{The model and assumptions}

\subsubsection{Admissible functions and domains}\label{subsubsecAssumptionsGeneral}
We now consider the free energy on a bounded domain, with electrostatic interactions. Let $\Omega\subset[0,2\pi]^3$ be a Lipschitz domain, so that there exists $\delta_1>0$ with $d(\Omega,\partial[0,2\pi])^3=\delta_1$. If $\Omega\not\subset[0,2\pi]^3$ but is still bounded, then we simply rescale and translate our domain. We consider Lipschitz subdomains $\Omega_\epsilon\subset\Omega$, so that there exists $\delta_\epsilon$ with $d(\Omega_\epsilon,\partial\Omega)>\delta_\epsilon$, and assume $c_7\epsilon^\alpha>\delta_\epsilon >c_6\epsilon^\alpha$ for some $1>\alpha > 0$. We have a boundary data, $b_0 \in W^{1,2}(\Omega,\mm)$, and work in the admissible set 
\begin{equation}
\mathcal{A}_\epsilon =\left\{ b \in L^\infty(\Omega,\mqc):(b_0-b)|_{\Omega\setminus\Omega_\epsilon}=0\right\}.
\end{equation}
Since $\Omega$ is a Lipschitz domain, we can extend $b_0$ to $\tilde{b}_0\in W^{1,2}([0,2\pi]^3,\mathbb{R}^k)$. Furthermore, since $\Omega$ and $\partial[0,2\pi]^3$ are separated, by multiplying by a smooth function which is zero $\mathbb{R}^3\setminus[0,2\pi]^3$ and unity on $\Omega$, we can assume that the extension is in $W^{1,2}_0([0,2\pi]^3,\mathbb{R}^k)$. Even more so, if $P_{\mqc}$ is the projection from $\mathbb{R}^k$ to $\mqc$, then $P_{\mqc}\tilde{b} \in W^{1,2}_0([0,2\pi]^3,\mqc)$, and $(P_{\mqc}\tilde{b}_0-b_0)|_{\Omega}=0$ Therefore we identify $\tilde{b}_0$ with a periodic extension, $P_{\mqc}\tilde{b}_0\in W^{1,2}(\mt^3,\mqc)$. 

With abuse of notation, we identify all functions $b \in \mathcal{A}_\epsilon$ with their extension $\tilde{b}\in W^{1,2}(\mt^3,\mathbb{R}^k)$ with $\tilde{b}\in \mt^3\setminus\Omega=P_{\mqc}\tilde{b}_0$. Importantly, the values outside of $\Omega_\epsilon$ are fixed.

We also impose a decay assumption on $g$. We assume there exists $C>0$ and $p>5$ with $(1-\alpha)(p-3)>2$ so that if $|z|$ is sufficiently large, then $g(z)\leq C|z|^{-p}$.

\subsubsection{Electrostatic interactions}\label{subsubsecElectrostatic}
We assume that the electrical anisotropy of the media is described by a Lipschitz, tensor-valued function function $A:\mqc \to \mathbb{R}^{3\times 3}$, so that the following holds:
\begin{itemize}
\item $A(b)=A(b)^T$ for all $b \in \mqc$.
\item We have that $\min\limits_{b \in \mqc} \lambda_{\min}(A(b))>0$. Since $A$ is continuous and $\mqc$ compact, $\max\limits_{b \in \mqc}\lambda_{\max}(A(b))<+\infty$ necessarily. 
\end{itemize}

Then the electrostatic contribution to the energy on a domain $\frac{1}{\epsilon}\Omega$ is given by 
\begin{equation}
-\frac{1}{2}\int_{\frac{1}{\epsilon}\Omega}D(x)\cdot E(x)\,dx,
\end{equation}
where $D$ is the displacement field, assumed to depend on the electric field as $D(x)=A(\tilde{b}(x))E(x)$ for the order parameter field $\tilde{b}\in L^\infty\left(\frac{1}{\epsilon}\Omega,\mathbb{R}^k\right)$. Maxwell's equations at equilibrium and in the absence of external charges gives that $\nabla \times E=0$, $\nabla \cdot D=0$. Therefore we can write this in terms of an electrostatic potential ${\tilde{\phi}} \in W^{1,2}\left(\frac{1}{\epsilon}\Omega,\mathbb{R}\right)$ subject to a Dirichlet condition ${\tilde{\phi}}|_{\frac{1}{\epsilon}\partial\Omega}={\tilde{\phi}}_0$, as 
\begin{equation}
\max\limits_{{\tilde{\phi}} \in {\tilde{\phi}}_0+W^{1,2}_0\left(\frac{1}{\epsilon}\Omega,\mathbb{R}\right)}-\frac{1}{2}\int_{\frac{1}{\epsilon}\Omega}A(\tilde{b}(x))\nabla {\tilde{\phi}}(x)\cdot \nabla {\tilde{\phi}}(x)\,dx. 
\end{equation}
In this case we see that $\nabla \cdot D=0$ is simply the Euler-Lagrange equation for this maximisation problem, which is satisfied by only the unique maximiser. By performing a change of variables $x\mapsto \epsilon y$, $\tilde{b}(x)=b(\epsilon x)$, $\tilde{\phi}(x)=\phi(\epsilon x)$, electrostatic term becomes 
\begin{equation}\label{eqElectrostaticTerm}
\frac{1}{\epsilon}\mathcal{E}(b)=\max\limits_{\phi \in \phi_0+W^{1,2}(\Omega,\mathbb{R})}-\frac{1}{2\epsilon}\int_{\Omega}A(b(y))\nabla \phi(y) \cdot \nabla \phi(y)\,dx. 
\end{equation}

This brings in a technical point in that our problem to be solved will no longer by a minimisation problem, but rather a saddle-point problem in $b$ and $\phi$. However, by viewing $\mathcal{E}$ as just a function of $b$, and hiding the dependence on $\phi$, we can overcome these issues. We have the following results that simplify the analysis

\begin{proposition}\label{propElectrostatic}
Let $\mathcal{E}(b)$ be as given in \eqref{eqElectrostaticTerm}. Then the following hold:
\begin{enumerate}
\item $\mathcal{E}(b)$ is uniformly bounded from below. 
\item If $b_k \overset{L^2}{\to} b$, then $\mathcal{E}(b_k)\to \mathcal{E}(b)$. That is, $\mathcal{E}$ is continuous with respect to strong $L^2$ convergence.  
\item If $\Phi_b$ denotes the unique solution of the maximisation problem in \eqref{eqElectrostaticTerm} for a given $b$, then if $b_k \overset{L^2}{\to}b$, $\Phi_{b_k}\overset{W^{1,2}}{\to}\Phi_b$. 
\end{enumerate}
\end{proposition}
\begin{proof}
These results follow from \cite[Appendix A]{cesana2009strain}, since $A$ is Lipschitz, implying that if $b_k \overset{L^2}{\to}b$, then $A(b_k)\overset{L^2}{\to}A(b)$.
\end{proof}

Thus the saddle-point energy, with $x$ dependence suppressed for brevity when unambiguous, is 
\begin{equation}
\begin{split}
\int_{\frac{1}{\epsilon}\Omega}\psi_s(\tilde b)\,dx-\frac{1}{2}\int_{\frac{1}{\epsilon}\Omega}\int_{\frac{1}{\epsilon}\Omega} K(x-y)\tilde b(x)\cdot\tilde b(y)\,dy\,dx -\frac{1}{2}\int_{\frac{1}{\epsilon}\Omega}A(\tilde b)\nabla \tilde \phi \cdot \nabla \tilde \phi \,dx.
\end{split}
\end{equation}
Performing a change of variables as we did in \eqref{eqChangeOfVariables}, this means we consider 
\begin{equation}
\begin{split}
\min\limits_{b}\max\limits_{\phi}&\frac{1}{\epsilon^3}\int_{\Omega}\psi_s(b)\,dx-\frac{1}{2\epsilon^6}\int_{\Omega}\int_{\Omega} K\left(\frac{x-y}{\epsilon}\right)b(x)\cdot b(y)\,dy\,dx \\
&-\frac{1}{2\epsilon}\int_{\Omega}A(b)\nabla \phi \cdot \nabla \phi \,dx\\
=\min\limits_{b}&\frac{1}{\epsilon^3}\int_{\Omega}\psi_s(b)\,dx-\frac{1}{2\epsilon^6}\int_{\Omega}\int_{\Omega} K\left(\frac{x-y}{\epsilon}\right)b(x)\cdot b(y)\,dy\,dx+\frac{1}{\epsilon}\mathcal{E}(b)
\end{split}
\end{equation}

We thus multiply through by $\epsilon$ as before and consider $\min\limits_{b \in \mathcal{A}_\epsilon} \mathcal{G}_\epsilon'(b)$, with 

\begin{equation}
\begin{split}
\mathcal{G}_\epsilon'(b)=&\frac{1}{\epsilon^2}\int_{\Omega}\psi(b(x))\,dx+\frac{1}{2\epsilon^5}\int_{\Omega}\int_{\Omega}K\left(\frac{x-y}{\epsilon}\right)b(y)\,dy\,dx+\mathcal{E}(b). 
\end{split}
\end{equation}

\subsection{$\Gamma$-convergence}\label{subsecGammaBoundedDomain}

With the model set up, we now turn to investigating the $\Gamma$-limit. The key idea is that given $b \in \mathcal{A}_\epsilon$, we can find an appropriate extension $\tilde{b}\in L^\infty(\mt^3,\mathbb{R}^k)$, so that $\mathcal{G}_\epsilon'(b)$ is asymptotically equivalent to $\mathcal{F}_\epsilon(\tilde{b})$ with a constraint. At this point, the result follows essentially from the results for periodic domains. The bulk of this subsection is devoted to showing that the ``error-term" is asymptotically small so that such an embedding into the periodic domain is successful.
Heuristically, the error term is directly related to the non-locality of the bilinear form. By partitioning the periodic domain into $\mt^3\setminus\Omega$ and $\Omega$,  the bilinear form obtains a cross-term involving an integral over $\mt^3\setminus \Omega \times \Omega$, which would not exist for a local energy functional. In order to show this cross-term does not effect the minimisation in an asymptotic sense, we exploit the fact that $K$ decays faster than $d(\mt^3\setminus\Omega,\Omega_\epsilon)$ approaches zero. We now proceed to find the appropriate estimates.

\begin{proposition}
There exists $\tilde{C}>0$ so that if $\epsilon>0$ is sufficiently small, then for all $x \in \Omega_\epsilon$, $\frac{1}{\epsilon^3}\int_{\mathbb{R}^3\setminus\Omega}\frac{1}{\epsilon^3}K\left(\frac{x-y}{\epsilon}\right)\,dx \leq \tilde{C}\epsilon^{(1-\alpha)(p-3)}I$. 
\end{proposition}
\begin{proof}
First note that if $x \in \Omega_\epsilon$, $d(x,\r3\setminus\Omega)>\delta_\epsilon\geq c_6\epsilon^\alpha$. Therefore 
\begin{equation}
\begin{split}
&\int_{\mathbb{R}^3\setminus\Omega}\frac{1}{\epsilon^3}K\left(\frac{x-y}{\epsilon}\right)\,dx\\
= & \int_{\frac{1}{\epsilon}(x-\mathbb{R}^3\setminus\Omega)}K(z)\,dz\\
\leq & M \int_{\r3\setminus B_{\frac{\delta_\epsilon}{\epsilon}}}g(z)I\,dz\\
\leq & \frac{4\pi MC}{3-p} \left[r^{3-p}\right]_{\frac{\delta_\epsilon}{\epsilon}}^\infty I\\
=& \frac{4\pi MC}{3-p} \left(\frac{\epsilon}{\delta_\epsilon}\right)^{p-3}I\leq \tilde{C}\epsilon^{(1-\alpha)(p-3)}I.
\end{split}
\end{equation}
The estimate on $g$ holds for $\frac{\delta_\epsilon}{\epsilon}\geq c_6\frac{1}{\epsilon^1-\alpha}$ sufficiently large, i.e. $\epsilon$ sufficiently small. 
\end{proof}

\begin{corollary}
There exists some remainder term $R^1_\epsilon:\mathcal{A}_\epsilon\to\mathbb{R}$ which tends to zero uniformly in its argument so that 
\begin{equation}
\frac{1}{2\epsilon^5}\int_{\Omega_\epsilon}\left(\int_{\Omega}K\left(\frac{x-y}{\epsilon}\right)\,dy \right)b(x)\cdot b(x)\,dx=\frac{1}{2\epsilon^2}\int_{\Omega_\epsilon}K_0b(x)\cdot b(x)\,dx+R^1_\epsilon(b),
\end{equation}
with $K_0=\int_{\mathbb{R}^3}K(z)\,dz$.
\end{corollary}
\begin{proof}
This follows since 
\begin{equation}
\begin{split}
&\left|\frac{1}{2\epsilon^5}\int_{\Omega_\epsilon}\left(\int_{\Omega}K\left(\frac{x-y}{\epsilon}\right)\,dy \right)b(x)\cdot b(x)\,dx-\int_{\Omega_\epsilon}K_0b(x)\cdot b(x)\,dx\right|\\
=&\left|\frac{1}{2\epsilon^5}\int_{\Omega_\epsilon}\left(\int_{\mathbb{R}^3\setminus\Omega}K\left(\frac{x-y}{\epsilon}\right)\,dy \right)b(x)\cdot b(x)\,dx\right|\\
\leq & |\Omega_\epsilon|\frac{||b||_\infty^2}{2\epsilon^2}\sup\limits_{x \in \Omega_\epsilon}\int_{\mathbb{R}^3\setminus\Omega}\frac{1}{\epsilon^3}\left|K\left(\frac{x-y}{\epsilon}\right)\right|\,dx\,dy\\
\leq &\tilde C|\Omega_\epsilon|\frac{||b||_\infty^2}{2}\epsilon^{(1-\alpha)(p-3)-2}.
\end{split}
\end{equation}
Since $\mqc$ is bounded and by assumption $(1-\alpha)(p-3)>2$ and $\mathcal{L}^3(\Omega_\epsilon)\leq \mathcal{L}^3(\Omega)$, the result follows.
\end{proof}

\begin{proposition}
There exists a positive constant $\tilde C>0$, so that for $\epsilon >0$ sufficiently small and $b \in \mathcal{A}_\epsilon$, 
\begin{equation}
\int_{\Omega}\int_{\Omega} \left(K_\epsilon(x-y)-\frac{1}{\epsilon^3}K\left(\frac{x-y}{\epsilon}\right)\right)(b(x)-b(y))\cdot (b(x)-b(y))\,dx\,dy\leq \tilde C\epsilon^{p-3}.
\end{equation}
\end{proposition}
\begin{proof}
Since $b$ is bounded it suffices to provide an estimate the term involving $K$. In this case, 
\begin{equation}
\begin{split}
&\int_{\Omega}\int_{\Omega} \left(K_\epsilon(x-y)-\frac{1}{\epsilon^3}K\left(\frac{x-y}{\epsilon}\right)\right)(b(x)-b(y))\cdot (b(x)-b(y))\,dx\,dy\\
\leq & 2||b||_\infty^2\int_{\Omega}\int_{\mt^3} \left|K_\epsilon(x-y)-\frac{1}{\epsilon^3}K\left(\frac{x-y}{\epsilon}\right)\right|\,dx\,dy\\
=& 2||b||_\infty^2\int_{\Omega}\int_{\mt^3} \left|\sum\limits_{k \in \mathbb{Z}^3\setminus\{0\}}\frac{1}{\epsilon^3}K\left(\frac{x-y+2\pi k}{\epsilon}\right)\right|\,dx\,dy\\
\leq & 2||b||_\infty^2\int_{\Omega}\int_{\mt^3} \sum\limits_{k \in \mathbb{Z}^3\setminus\{0\}}\frac{1}{\epsilon^3}\left|K\left(\frac{x-y+2\pi k}{\epsilon}\right)\right|\,dx\,dy\\
=& 2||b||_\infty^2\int_{\Omega}\int_{\mathbb{R}^3\setminus \mt^3}\frac{1}{\epsilon^3}\left|K\left(\frac{x-y}{\epsilon}\right)\right|\,dx\,dy\\
\leq & 2M||b||_\infty^2\int_{\Omega}\int_{\mathbb{R}^3\setminus\mt^3}\frac{1}{\epsilon^3}g\left(\frac{x-y}{\epsilon}\right)\,dx\,dy\\
\leq & 2M||b||_\infty^2\int_{\Omega}\int_{\frac{1}{\epsilon}(x-\mathbb{R}^3\setminus\mt^3)}g\left(z\right)\,dx\,dy\\
\end{split}
\end{equation}
Now we note that for all $x \in \Omega$, $\r3\setminus B_{\frac{\delta_1}{\epsilon}}\subset \frac{1}{\epsilon}(x-\mathbb{R}^3\setminus\mt^3)$. Then the remaining integral can be estimated as in the previous result, although it is somewhat simpler since $\Omega$ is uniformly bounded away from $\r3\setminus[0,2\pi]^3$. This gives that for small $\epsilon>0$ the integral can be bounded by 
\begin{equation}
\tilde{C}|\Omega|\epsilon^{p-3}.
\end{equation}
\end{proof}

\begin{definition}
We define the second remainder $R_\epsilon^2:\mathcal{A}_\epsilon\to\mathbb{R}$ by 
\begin{equation}
R_\epsilon^2(b)=\frac{1}{\epsilon^2}\int_{\Omega}\int_{\Omega} \left(K_\epsilon(x-y)-\frac{1}{\epsilon^3}K\left(\frac{x-y}{\epsilon}\right)\right)(b(x)-b(y))\cdot (b(x)-b(y))\,dx\,dy.
\end{equation}
\end{definition}
Note that $R_\epsilon^2\to 0$ uniformly by the previous result. 

\begin{proposition}
There is a constant $\tilde C>0$ so that for $\epsilon>0$ sufficiently small and all admissible $b \in \mathcal{A}_\epsilon$, 
\begin{equation}
\frac{1}{\epsilon^2}\int_{\Omega_\epsilon}\int_{\mt^3\setminus\Omega}K_\epsilon(x-y)\cdot(b(x)-b(y))\cdot (b(x)-b(y))\,dx\,dy\leq \tilde C\epsilon^{(1-\alpha)(p-3)-2}.
\end{equation}
\end{proposition}
\begin{proof}
The heuristics follow almost identically to the previous arguments, so for brevity only a sketch is provided for the estimate. 
\begin{equation}
\begin{split}
&\frac{1}{\epsilon^2}\int_{\Omega_\epsilon}\int_{\mt^3\setminus\Omega}K_\epsilon(x-y)(b(x)-b(y))\cdot (b(x)-b(y))\,dx\,dy\\
\leq & \frac{2||b||_\infty^2 M}{\epsilon^2}\int_{\Omega_\epsilon} \int_{\r3\setminus\Omega} g(x-y)\,dx\,dy=O(\epsilon^{(1-\alpha)(p-3)-2})
\end{split}
\end{equation}
by the same argument as before, and tends to zero uniformly in $b$.
\end{proof}

\begin{definition}
Define the third remainder, $R_\epsilon^3:\mathcal{A}_\epsilon\to\mathbb{R}$ by 
\begin{equation}
R_\epsilon^3(b)= \frac{1}{\epsilon^2}\int_{\Omega_\epsilon}\int_{\mt^3\setminus\Omega}K_\epsilon(x-y)(b(x)-b(y))\cdot (b(x)-b(y))\,dx\,dy.
\end{equation}
\end{definition}
Note that under our assumptions, $R_\epsilon^3 \to 0$ uniformly as $\epsilon \to 0$.

\begin{definition}
Let $U_1,U_2 \subset\mathbb{R}^3$ be measurable, and $b \in \mathcal{A}_\epsilon$. Define the bilinear form $B_\epsilon$ as
\begin{equation}
B_\epsilon(b,U_1,U_2)=\int_{U_1}\int_{U_2} K_\epsilon(x-y)(b(x)-b(y))\cdot (b(x)-b(y))\,dx\,dy.
\end{equation}
\end{definition}

\begin{theorem}\label{theoremEquivalentEnergies}
There exists a function $R_\epsilon : \mathcal{A}_\epsilon\to\mathbb{R}$ which tends to zero uniformly as $\epsilon \to 0$, and constants $m_\epsilon$ so that if $c_5=\inf\limits_{b \in \mq}\psi_s(b)-\frac{1}{2}K_0b\cdot b$, then
\begin{equation}
\begin{split}
\mg'_\epsilon(b)=& \frac{1}{\epsilon^2}\int_{\Omega_\epsilon}\psi_s(b(x))-\frac{1}{2}K_0b(x)\cdot b(x)-c_5\,dx\\
&+\frac{1}{4\epsilon^2}\int_{\mt^3}\int_{\mt^3}K_\epsilon(x-y)(b(x)-b(y))\cdot (b(x)-b(y))\,dx\,dy+\mathcal{E}(b)+R_\epsilon(b)+m_\epsilon.  
\end{split}
\end{equation}
\end{theorem}
\begin{proof}
For the bulk term, we we immediately have 
\begin{equation}
\begin{split}
&\int_{\Omega}\psi_s(b(x))-\frac{1}{2}\left(\frac{1}{\epsilon^3}\int_{\Omega}K\left(\frac{x-y}{\epsilon}\right)\,dy\right)b(x)\cdot b(x)\,dx\\
=& \int_{\Omega_\epsilon} \psi_s(b(x))-\frac{1}{2}K_0b(x)\cdot b(x)\,dx+m_\epsilon^1 +R_\epsilon^1(b).
\end{split}
\end{equation}
The constant $m_\epsilon^1=\int_{\Omega\setminus\Omega_\epsilon }K_0b(x)\cdot b(x)\,dx$ depends only on the boundary data $b|_{\Omega\setminus\Omega_\epsilon}$. 

To deal with the non-local term, first estimate this as 
\begin{equation}
\begin{split}
&\frac{1}{\epsilon^2}\int_{\Omega}\int_{\Omega} \frac{1}{\epsilon^3}K\left(\frac{x-y}{\epsilon}\right)(b(x)-b(y))\cdot (b(x)-b(y))\,dx\,dy\\
=& \frac{1}{\epsilon^2}\int_{\Omega}\int_{\Omega} K_\epsilon(x-y)(b(x)-b(y))\cdot (b(x)-b(y))\,dx\,dy+R_\epsilon^2(b)\\
=& \frac{1}{\epsilon^2}B_\epsilon(b,\Omega,\Omega) + R_\epsilon^2(b)
\end{split}
\end{equation}

Then we decompose $B_\epsilon(b,\Omega,\Omega)$ as 
\begin{equation}
\begin{split}
B_\epsilon(b,\Omega,\Omega)=&B_\epsilon(b,\Mt3,\Mt3)-2B_\epsilon(b,\Mt3\setminus\Omega,\Omega)-B_\epsilon(b,\Mt3\setminus \Omega,\Mt3\setminus\Omega)\\
=&B_\epsilon(b,\Mt3,\Mt3)-2B_\epsilon(b,\Mt3\setminus\Omega,\Omega_\epsilon)-2B_\epsilon(\Mt3\setminus\Omega,\Omega\setminus\Omega_\epsilon)-B_\epsilon(b,\Mt3\setminus \Omega,\Mt3\setminus\Omega)\\
=& \bigg[B_\epsilon(b,\Mt3,\Mt3)-2B_\epsilon(b,\Mt3\setminus\Omega,\Omega_\epsilon)\bigg]\\
&-\bigg[2B_\epsilon(\Mt3\setminus\Omega,\Omega\setminus\Omega_\epsilon)+B_\epsilon(b,\Mt3\setminus \Omega,\Mt3\setminus\Omega)\bigg]\\
=& B_\epsilon(b,\Mt3,\Mt3)-2R_3^\epsilon(b)-m^2_\epsilon,
\end{split}
\end{equation}
with $m^2_\epsilon=2B_\epsilon(\Mt3\setminus\Omega,\Omega\setminus\Omega_\epsilon)+B_\epsilon(b,\Mt3\setminus \Omega,\Mt3\setminus\Omega)$ only dependent on the boundary values. 

We now combine these three results, to give that 
\begin{equation}
\begin{split}
&\frac{1}{\epsilon^2}\int_{\Omega}\psi_s(b(x))\,dx-\frac{1}{2\epsilon^5}\int_{\Omega}\int_{\Omega}K\left(\frac{x-y}{\epsilon}\right)b(x)\cdot b(y)\,dx\,dy\\
=&\frac{1}{\epsilon^2}\int_{\Omega_\epsilon}\psi_s(b(x))-K_0b(x)\cdot b(x)-c_5\,dx-c_5|\Omega_\epsilon| +m_\epsilon^1 R^1_\epsilon(b)\\
&+ \frac{1}{\epsilon^2}B_\epsilon(b,\Omega,\Omega)+R_\epsilon^2(b)\\
=&\frac{1}{\epsilon^2}\int_{\Omega_\epsilon}\psi_s(b(x))-K_0b(x)\cdot b(x)-c_5\,dx-c_5|\Omega_\epsilon| +m_\epsilon^1+ R^1_\epsilon(b)\\
&+ \frac{1}{\epsilon^2}B_\epsilon(b,\Mt3,\Mt3)-2\epsilon^2R_\epsilon^3(b)-\frac{m^2_\epsilon}{\epsilon^2}+R_\epsilon^2(b).
\end{split}
\end{equation}
Therefore by defining 
\begin{equation}
\begin{split}
m_\epsilon = & m^1_\epsilon-c_5|\Omega_\epsilon|-\frac{m^2_\epsilon}{\epsilon^2},\\
R_\epsilon=&R_\epsilon^1-2R_\epsilon^3+R_\epsilon^2,
\end{split}
\end{equation}
the result holds. 
\end{proof}

In light of this result, it suffices to consider the $\Gamma$-limit of the asymptotically equivalent energy,
\begin{equation}
\begin{split}
\mg_\epsilon(b)=&\frac{1}{\epsilon^2}\int_{\Omega_\epsilon}\psi(b(x))\,dx+\frac{1}{2\epsilon^2}\int_{\mt^3}\int_{\mt^3}K_\epsilon(x-y)\cdot (b(x)-b(y))^{\otimes 2}\,dx\,dy+\mathcal{E}(b)
\end{split}
\end{equation}

\begin{proposition}
Assume that $b_{\epsilon}\in \mathcal{A}_\epsilon$ with $\mg_\epsilon(b_\epsilon)$ uniformly bounded. Then there is a subsequence $b_{\epsilon_j}$ so that $b_{\epsilon_j}\overset{L^2}{\to}b \in W^{1,2}(\mt^3,\mathbb{R}^k)$, with $b(x) \in \mm$ for almost every $x \in \Omega$. In particular, $b_{\epsilon_j}\overset{L^2}{\to}b \in W^{1,2}(\Omega,\mm)$ with $b|_{\partial\Omega}=b_0|_{\partial\Omega}$. 
\end{proposition}
\begin{proof}

Since the electrostatic and bulk terms admit lower bounds, this implies that 
\begin{equation}
\frac{1}{\epsilon^2}\int_{\mt^3}\int_{\mt^3}K_\epsilon(x-y)\cdot(b_{\epsilon}(x)-b_{\epsilon}(y))^{\otimes 2}\,dx\,dy
\end{equation}
is also uniformly bounded. Therefore by \Cref{corollaryW12Limit}, we can extract a subsequence $b_{\epsilon_j} \overset{L^2}{\to}b \in W^{1,2}(\mt^3,\mathbb{R}^k)$. Furthermore, for $x \in \Omega\setminus \Omega_\epsilon$, $b(x)\in \mm$ almost everywhere. Since $\mm =\psi^{-1}(0)$, this gives 
\begin{equation}
\begin{split}
\frac{1}{{\epsilon_j}^2} \int_{\Omega_{\epsilon_j}} \psi(b_{\epsilon_j}(x))\,dx =&\frac{1}{{\epsilon_j}^2}\int_{\Omega}\psi(b_{\epsilon_j}(x))\,dx
\end{split}
\end{equation}
which is bounded. Thus since the energy is bounded, this implies that $\psi(b_{\epsilon_j}(x))\to 0$ almost everywhere in $\Omega$, so taking a pointwise converging subsequence proves $b(x)\in \mm$ almost everywhere in $\Omega$. In the sense of traces, $b|_{\partial\Omega}=b_0|_{\partial\Omega}$, since $b-b_0$ is a $W^{1,2}$ function supported on the closure of $\Omega$.
\end{proof}

Before turning to the $\Gamma$-convergence result, we establish a lemma required for our construction of recovery sequences.

\begin{lemma}\label{lemmaLipschitzSubdomain}
Let $U\subset \mathbb{R}^n$ be a bounded Lipschitz domain. Then for every $\epsilon>0$ there exists a compact set $A\subset \Omega$ so that $\mathcal{L}^n(U\setminus A)<\epsilon$ and $U\setminus A$ is Lipschitz. 
\end{lemma}
\begin{proof}
The result is essentially Vitali's covering theorem. We denote $x+\overline{B_r}$ the closed ball of radius $r>0$ centred at $x\in\r3$. Since $U$ is open, we can define a function $R:U\to(0,\infty)$ so by $R(x)=\frac{1}{2}\sup\{r>0 : x +B_r\subset\Omega\}$. Then if $r\leq R(x)$, $x+\overline{B_r}$ is compactly supported in $U$. Furthermore $\{x+\overline{B_r}:r<R(x), x \in U\}$ is a Vitali covering of $U$. Therefore we can take a countable pairwise disjoint subcollection $(x_i +\overline{B_{r_i}})_{i \in \mathbb{N}}$, whose union covers $U$ up to a set of measure zero. Even more so if $A_N=\bigcup\limits_{i=1}^N(x_i+\overline{B_{r_i}})$, then $A_N$ is a finite union of disjoint closed balls, compactly supported in $U$, with $\lim\limits_{N \to \infty}\mathcal{L}^n(U\setminus A_N)\to 0$. Thus by taking $A=A_N$ for sufficiently large $N$, we have that $\mathcal{L}^n(U\setminus A)<\epsilon$. Furthermore, since $A_N$ is a union of disjoint balls compactly supported in $U$, we have that $\partial (U\setminus A)=\partial U \cup \left(\bigcup\limits_{i=1}^N x_i +\partial B_{r_i}\right)$, which is a disjoint union of Lipschitz surfaces. Therefore $U\setminus A$ is a Lipschitz domain.

\end{proof}

\begin{theorem}
The functionals $\mathcal{G}_\epsilon \overset{\Gamma}{\to}\mathcal{G}$, where 
\begin{equation}
\mathcal{G}(b)=\int_{\mt^3}\frac{1}{4}L\nabla b(x)\cdot \nabla b(x)\,dx +\frac{1}{2}\int_{\Omega}A(b)^{-1}D(x)\cdot D(x)\,dx +\mathcal{E}(b)
\end{equation}
if $b \in W^{1,2}(\mt^3,\mathbb{R}^k)$ with $b(x) \in \mm$ for almost every $x \in \Omega$ and $b=b_0$ on $\mt^3\setminus\Omega$, and is $+\infty$ otherwise. The mode of convergence is $L^2$ strong convergence. Furthermore let $b_\epsilon \overset{L^2}{\to} b$, and denote the solutions of the maximisation problem defining $\mathcal{E}(b_\epsilon)$ as $\Phi_\epsilon$. Then $\text{div}(A(b_\epsilon)\Phi_\epsilon)=0$ and $\Phi_\epsilon|_{\partial\Omega}=\phi_0$. Then $\Phi_\epsilon\overset{W^{1,2}}{\to}\Phi$, where $\text{div}(A(b)\nabla \Phi)=0$  and $\Phi|_{\partial\Omega}=\phi_0$.  
\end{theorem}
\begin{proof}
The electrostatic term $\mathcal{E}(b)$ is continuous with $L^2$ convergence, so does not effect the result. The liminf inequality follows directly by the same argument as \Cref{theoremGammaPeriodic}. For the limsup inequality, we can no longer use constant recovery sequences, since for $b \in W^{1,2}(\mt^3,\mathbb{R}^k)$ with $b-b_0=0$ on $\mt^3\setminus\Omega$, it is not in general true that $b|_{\Omega\setminus\Omega_\epsilon}=b_0$. We overcome this with an interpolation by harmonic functions.

Let $b \in W^{1,2}(\mt^3,\mathbb{R}^k)$ with $b(\Omega)\subset \mm$ and $b|_{\mt^3\setminus\Omega}=b_0$. Let $E_\epsilon \subset \Omega_\epsilon$ be a compact set, so that $U_\epsilon=\Omega_\epsilon\setminus E_\epsilon$ satisfies $\mathcal{L}^3(U_\epsilon)\leq c_8 \epsilon^3$ with $U_\epsilon \setminus E_\epsilon$ Lipschitz, which exists by \Cref{lemmaLipschitzSubdomain}. Let $b_\epsilon (x)=b(x)$ for $x \in \mt^3\setminus U_\epsilon$, and let $b_\epsilon$ satisfy 
\begin{equation}\label{eqHarmonicReplacement}
\begin{array}{r l l }
\int_{U_\epsilon} \nabla b_\epsilon(x) \cdot \nabla u(x)\,dx= & 0 &(\forall u \in W^{1,2}_0(U_\epsilon)),\\
b_\epsilon(x)= & b(x) &(x \in \partial U_\epsilon).
\end{array}
\end{equation}
Since $U_\epsilon$ is Lipschitz and $b-b_\epsilon \in W^{1,2}_0(U_\epsilon,\mathbb{R}^k)$, the extension by zero is in $W^{1,2}(\mt^3)$, therefore $b \in W^{1,2}(\mt^3,\mathbb{R}^k)$. First we show $b_\epsilon \overset{W^{1,2}}{\to} b$. The maximum principle gives that $b_\epsilon$ admits a uniform bound, so $b_\epsilon$ and $b$ are both bounded and only differ on $U_\epsilon$. This set has which has vanishing measure as $\epsilon \to 0$, so it holds that $b_\epsilon \overset{L^2}{\to} b$. To show that the gradients converge, we use that $b_\epsilon$ is harmonic and $b-b_\epsilon$ can be used to test the weak form of the PDE in \eqref{eqHarmonicReplacement}. Suppressing the $x$ dependence for readability, 
\begin{equation}
\begin{split}
||\nabla b_\epsilon - \nabla b||_2^2 =&\int_{\mt^3}|\nabla b_\epsilon-\nabla b|^2\,dx\\
=&\int_{ U_\epsilon}\left(\nabla b_\epsilon-\nabla b\right)\cdot \left(\nabla (b_\epsilon-b)\right)\,dx\\
=& \int_{ U_\epsilon}\nabla b\cdot \left(\nabla b_\epsilon-\nabla b\right)\,dx\\
\leq & \left(\int_{U_\epsilon}|\nabla b|^2\,dx\right)^\frac{1}{2} ||\nabla b_\epsilon - \nabla b||_2\\
\Rightarrow ||\nabla b_\epsilon - \nabla b ||_2 \leq & \left(\int_{U_\epsilon}|\nabla b|^2\,dx\right)^\frac{1}{2}.
\end{split}
\end{equation}
Since the measure of $U_\epsilon$ tends to zero and $b \in W^{1,2}$, the integral on the right hand side tends to zero as $\epsilon \to 0$, so that $\nabla b_\epsilon \overset{L^2}{\to} \nabla b $ and $b_\epsilon \overset{W^{1,2}}\to b$. Therefore by \Cref{propLimsup}, we have that 
\begin{equation}
\lim\limits_{\epsilon \to 0} \frac{1}{\epsilon^2}\int_{\mt^3}\int_{\mt^3}K_\epsilon(x-y)\cdot\left(b_\epsilon(x)-b_\epsilon(y)\right)^{\otimes 2}\,dx\,dy=\int_{\mt^3}L\nabla b(x)\cdot \nabla b(x)\,dx. 
\end{equation}
It remains to show that the bulk energy of $b_\epsilon$ tends to zero. Again since $ b_\epsilon$ is harmonic we have a maximum principle, and since the boundary data is in $\mm$, we have $b_\epsilon(x)\in \text{Conv}(\mm)$ almost everywhere. We can now obtain an estimate on $\psi_b(b(x))$. First, we note that since $\psi_b$ blows up uniformly at $\partial\mq$, we must have that $\mm$ is compactly supported in $\mq$. Then given $b_\epsilon(x)\in \text{Conv}(\mm)\subset\mq$, using the convexity of $\psi_s$ and non-negativity of $K_0$ we have 
\begin{equation}
\psi(b_\epsilon(x))=\psi_s(b(x))-\frac{1}{2}K_0b(x)\cdot b(x)\leq \max\limits_{\mm}\psi_s=c_9,
\end{equation}
which is a uniform bound on $\psi_b(b_\epsilon)$. Therefore 
\begin{equation}
\begin{split}
\frac{1}{\epsilon^2}\int_{\Omega}\psi_b(b_\epsilon)\,dx=& \int_{U_\epsilon}\psi_s(b_\epsilon)\,dx\\
\leq & \frac{c_9}{\epsilon^2}\int_{U_\epsilon}\,dx\\
\leq & \frac{c_8c_9}{\epsilon^2}\epsilon^3=O(\epsilon).
\end{split}
\end{equation}
Therefore $\lim\limits_{\epsilon \to 0}\frac{1}{\epsilon^2}\int_{\Omega}\psi_b(b_\epsilon)\,dx\to 0$. Combining these, we have $b_\epsilon \to b$, and $\mg_\epsilon(b_\epsilon)\to \mg(b)$, giving the limsup inequality. The results for the convergence of the electrostatic potential are in \Cref{propElectrostatic}.
\end{proof}

\begin{remark}
Whilst the energy contains an elastic term given by the integral over $\mt^3\setminus \Omega$, this is only dependent on the prescribed boundary conditions extension and thus gives a constant which is irrelevant to the minimisation process. In fact we could subtract this constant from $\mg_\epsilon(b)$ to remove this. The condition that $b|_{\mt^3\setminus\Omega}=b_0$ then reduces to a Dirichlet condition on the boundary $\partial\Omega$ as $\mathcal{L}^3(\Omega\setminus\Omega_\epsilon)\to 0$. This gives that the minimisation problem for the $\Gamma$-limit is equivalent to minimising 
\begin{equation}
\frac{1}{4}\int_{\Omega}L\nabla b(x)\cdot \nabla b(x)\,dx
\end{equation}
over $b\in W^{1,2}(\Omega,\mm)$ with $b|_{\partial\Omega}=b_0|_{\partial\Omega}$.
\end{remark}

\section{Application to the Oseen-Frank model}

For definiteness, we now consider a case which will reduce to the Oseen-Frank elastic model. In this case, we take $\mn = \s2$ and $a(p)=p\otimes p-\frac{1}{3} \in \sm3$. We require our interaction kernel to be frame indifferent, so that $K(Rz)a(Rp)\cdot a(Rq)=K(z)a(p)\cdot a(q)$ for all $R \in \so3$. We can see that this highly constrains which bilinear forms are available \cite{smith1971isotropic}, we only have 
\begin{equation}\label{eqLondonDispersion}
K(z)a(p)\cdot a(q)=g_1(z)a(p)\cdot a(q)+g_2(z)(a(p)\h{z})\cdot (a(q)\h{z})+g_3(z)(\h{z}\cdot a(p)\h{z})(\h{z}\cdot a(q)\h{z}),
\end{equation}
where $g_i$ are frame indifferent functions on $\r3$. We assume that $g_i$ have sufficient integrability and decay for our results to hold. It is not necessary that $g_2,g_3$ are non-negative, however $g_1$ will have to be able to compensate. We write this in indices as
\begin{equation}
K(z)_{i_1i_2,j_1j_2}=g_1(z)\delta_{i_1,j_1}\delta_{i_2,j_2}+g_2(z)\delta_{i_1,j_1}\h{z}_{i_2}\h{z}_{j_2}+g_3(z)\h{z}_{i_1}\h{z}_{i_2}\h{z}_{j_1}\h{z}_{j_2},
\end{equation}
with
\begin{equation}
K(z)Q^1\cdot Q^2=\sum\limits_{i_1,i_2,j_1,j_2=1}^3K(z)_{i_1i_2,j_1j_2}Q_{i_1i_2}^1Q_{j_1j_2}^2.
\end{equation}

We note that $K$ admits a symmetry, so that if $\sigma$ is a permutation of $\{1,2,3\}$, then $K(z)_{ij,kl}=K(z)_{\sigma_i\sigma_j,\sigma_k\sigma_l}$. Furthermore $K_{i_1i_2,j_1j_2}=K_{j_1j_2,i_1i_2}$.

For the sake of simplifying later integrals involving moments of isotropic functions, we include the following lemma. 

\begin{lemma}
Let $g \in L^1(\r3)$ be isotropic and have finite fourth moment. Then 
\begin{equation}
\int_{\mathbb{R}^3}\left( z_1^4-3z_1^2z_2^2\right) g(z),dz=0.
\end{equation}
\end{lemma}
\begin{proof}
In order to see this, we note that due to the symmetry of indices, we can write
\begin{equation}
\int_{\mathbb{R}^3}\left( z_1^4-3z_1^2z_2^2 \right)g(z)\,dz=\frac{1}{2}\int_{\mathbb{R}^3} \left(z_1^4-6z_1^2z_2^2+z_2^4\right) g(z)\,dz.
\end{equation}
Next, we see that if $P(z)=z_1^4-6z_1^2z_2^2+z_2^4$, and $R$ is the rotation about $e_3$ of $\frac{\pi}{4}$, then $P(Rz)=-P(z)$ for all $z \in \r3$. This is just an algebraic computation, 
\begin{equation}
\begin{split}
P(Rz)=&\left(\frac{\sqrt{2}}{2}z_1-\frac{\sqrt{2}}{2}z_2\right)^4-6\left(\frac{\sqrt{2}}{2}z_1-\frac{\sqrt{2}}{2}z_2\right)^2\left(\frac{\sqrt{2}}{2}z_1+\frac{\sqrt{2}}{2}z_2\right)^2+\left(\frac{\sqrt{2}}{2}z_1+\frac{\sqrt{2}}{2}z_2\right)^4\\
=&\left(\frac{1}{2}z_1^2+\frac{1}{2}z_2^2-z_1z_2\right)^2-6\left(\frac{1}{2}z_1^2+\frac{1}{2}z_2^2+z_1z_2\right)\left(\frac{1}{2}z_1^2+\frac{1}{2}z_2^2-z_1z_2\right)\\
&+\left(\frac{1}{2}z_1^2+\frac{1}{2}z_2^2+z_1z_2\right)^2\\
=& \frac{1}{4}z_1^4-z_1^3z_2+\frac{3}{2}z_1^2z_2^2-z_1z_2^3+\frac{1}{4}z_2^4\\
&-\frac{3}{2}z_1^2+3z_1^2z_2^2-\frac{3}{2}z_2^2\\
&+ \frac{1}{4}z_1^4+z_1^3z_2+\frac{3}{2}z_1^2z_2^2+z_1z_2^3+\frac{1}{4}z_2^4\\
=& -z_1^2+6z_1^2z_2^2-z_2^2=-P(z).
\end{split}
\end{equation}
Thus by applying the change of variables $z\mapsto Rz$ into the integral, 
\begin{equation}
\begin{split}
&\int_{\mathbb{R}^3}\left( z_1^4-3z_1^2z_2^2\right) g(z)\,dz\\
=&\frac{1}{2}\int_{\mathbb{R}^3}P(z) g(z)\,dz\\
=& \frac{1}{2}\int_{\mathbb{R}^3}P(Rz)g(Rz)\,dz\\
=& -\frac{1}{2}\int_{\mathbb{R}^3}P(z)g(z)\,dz,
\end{split}
\end{equation}
so the integral is its own negative and hence zero.
\end{proof}

\begin{proposition}
Let $g \in L^1(\mathbb{R}^3)$ be isotropic and $Q \in \sm3$. Then 
\begin{equation}
\begin{split}
\int_{\r3}g(z)\h{z}_i\h{z}_j\,dz =&\delta_{ij}\int_{\r3}g(z)\h{z}_1^2\,dz,\\
\int_{\r3}g(z)(\h{z}\cdot Q\h{z})^2\,dz =& \frac{2}{3}\int_{\r3}g(z)\h{z}_1^4\,dz|Q|^2.
\end{split} 
\end{equation}
\end{proposition}
\begin{proof}
The first equality is more straightforward. If $i\neq j$, then the integral vanishes since it is odd in one of the components. If $i=j$, then the integrand is dependent only on $i$ since applying a rotation that swaps between labelling bases leaves the integral unchanged.

For the second inequality, we have that, as a map from $\sm3$ to $\mathbb{R}$, $\int_{\r3}g(z)(z\cdot Qz)^2\,dz$ is immediately frame indifferent and quadratic in $Q$. Therefore it must be of the form $c_1\text{Tr}(Q)^2+c_2\text{Tr}(Q^2)$. Since $Q$ is trace free however, it suffices to evaluate $c_2$. In this case, we test $Q=e_1\otimes e_1-e_2\otimes e_2$. Then $Q\h{z}\cdot \h{z} = \h{z}_1^2-\h{z}_2^2$ and $|Q|^2=2$. Thus
\begin{equation}
\begin{split}
2c_1=&\int_{\r3}(\h{z}_1^2-\h{z}_2^2)^2g(z)\,dz\\
=& \int_{\r3}\left(\h{z}_1^4+\h{z}_2^4-2\h{z}_1^2\h{z}_2^2\right)\,dz\\
=& 2\int_{\r3}(\h{z}_1^4-\h{z}_1^2z_2^2)g(z)\,dz\\
=& \frac{4}{3}\int_{\r3}\h{z}_1^4g(z)\,dz
\end{split}
\end{equation}
using the previous lemma. 
\end{proof}

\begin{corollary}\label{corollaryk0}
For $K$ as given in \eqref{eqLondonDispersion}, and $K_0=\int_{\r3}K(z)\,dz$, then $K_0Q\cdot Q=k_0|Q|^2$, where
\begin{equation}
k_0=\int_{\r3}g_1(z)\,dz+\int_{\r3}g_2(z)\h{z}_1^2\,dz + \frac{2}{3}\int_{\r3}g_3(z)\h{z}_1^4\,dz
\end{equation}
\end{corollary}
\begin{proof}
This follows immediately from the previous proposition. 
\end{proof}

In this case the bulk energy is given by $\psi_s(Q)-\frac{k_0}{2}|Q|^2-c_5$. Depending on the value of $k_0$, this is minimised either at $Q=0$, or $Q=s^*\left(n\otimes n-\frac{1}{3}I\right)$ for any $n \in \s2$ and $s^*>0$ depending on $k_0$ \cite{fatkullin2005critical}. We consider only the latter case. In this case with a nematic ground state, the minimising manifold $\mm$ is readily identified with projective space. By identifying $\mm$ with $\mathbb{R}P^2$, if a lifting for $Q \in W^{1,2}(\Omega,\mm)\to W^{1,2}(\Omega,\s2)$ exists, we say $Q$ is {\it orientable} and obtain the Oseen-Frank energy for a director field $n\in W^{1,2}(\Omega,\s2)$. This is always possible, for example, when the boundary data $Q_0$ is orientable and $\Omega$ is simply connected \cite{ball2008orientable}.

\begin{proposition}\label{propFrankConstants}
Let $Q=s^*\left(n\otimes n-\frac{1}{3}I\right)$ for some $n \in W^{1,2}(\Omega,\s2)$. For $K$ as given in \eqref{eqLondonDispersion}, the elastic component of the energy can be written 
\begin{equation}
L\nabla Q \cdot \nabla Q =K_1 (\text{div }n)^2 + K_2 |n\times \text{curl }n|^2 + K_3(n\cdot \text{curl }n)^2
\end{equation}
with
\begin{equation}
\begin{split}
\frac{1}{(s^*)^2}K_1=&2G^1_{1,0,0}+G^2_{1,1,0}+G^2_{2,0,0}+G^3_{3,0,0}-G^3_{2,1,0},\\
\frac{1}{(s^*)^2}K_2=& 2\left(G^1_{1,0,0}+G^2_{1,1,0}+G^3_{2,1,0}-G^3_{1,1,1}\right),\\
K_3=& K_1.
\end{split}
\end{equation}
The constants $G^n_{ijk}$ are defined by 
\begin{equation}
G^n_{ijk}=\int_{\r3}\frac{1}{|z|^{2(n-1)}}g_I(z)z_1^{2i}z_2^{2j}z_3^{2k}\,dz.
\end{equation}
\end{proposition}
\begin{proof}
Denote the components of $\nabla Q$ as $Q_{\alpha\beta,\gamma}$. Then the bilinear form is 
\begin{equation}
\begin{split}
L\nabla Q \cdot \nabla Q = &\int_{\r3}K(z)_{i_1i_2,j_1j_2}z_{i_3}z_{j_3}Q_{i_1i_2,i_3}Q_{j_1j_2,j_3}\,dz.
\end{split}
\end{equation}
It is immediate to see that this is a frame indifferent function of $\nabla Q$, following from the isotropy of $K$. In fact this holds for generally any $Q \in W^{1,2}(\Omega,\sm3)$, not just those with values in $\mm$ which are orientable. In this case, by \cite{mori1999multidimensional} we know that the bilinear form can be written as 
\begin{equation}
L_1Q_{\alpha\beta,\gamma}Q_{\alpha\beta,\gamma}+L_2Q_{\alpha\beta,\beta}Q_{\alpha\gamma,\gamma}+L_3Q_{\alpha\beta,\gamma}Q_{\alpha\gamma,\beta}.
\end{equation}
In the case when $Q$ gives an orientable line field, so $Q=s^*\left(n\otimes n-\frac{1}{3}I\right)$ for $n \in W^{1,2}(\Omega,\s2)$, we have that 
\begin{equation}
L\nabla Q \cdot \nabla Q = K_1 (\text{div }n)^2 + K_2 |n\times \text{curl }n|^2 + K_3(n\cdot \text{curl }n)^2, 
\end{equation}  
where the Frank constants $K_i$ are related to $L_i$ as in \cite{ballCambridgeLectures} by 
\begin{equation}
\begin{split}
\frac{1}{(s^*)^2}K_1=&2L_1+L_2+L_3,\\
\frac{1}{(s^*)^2}K_2=& 2L_1,\\
K_3=&K_1.
\end{split}
\end{equation}
It thus suffices to test values of $\nabla Q$ to obtain relations between the tensor $L$ and constants $L_1,L_2,L_3$, which then give the Frank constants.
First we try $Q(x)=(e_1\otimes e_1-e_2\otimes e_2)x_3$, so $\nabla Q = (e_1\otimes e_1  - e_2 \otimes e_2 )\otimes e_3$. Then $Q_{\alpha\beta,\gamma}Q_{\alpha\gamma,\beta}=Q_{\alpha\beta,\beta}Q_{\alpha\gamma,\gamma}=0$ and $Q_{\alpha\beta,\gamma}Q_{\alpha\beta,\gamma}=2$. Then 
\begin{equation}
\begin{split}
\frac{1}{(s^*)^2}K_2=&2L_1\\
=&L\nabla Q \cdot \nabla Q \\
=& \int_{\r3}K(z)_{1111}z_{3}^2-2K(z)_{1122}z_3^2+K(z)_{2222}z_3^2\,dz\\
=& 2\int_{\r3}K_{1111}z_3^2-K(z)_{1122}z_3^2\,dz\\
=& 2\int_{\r3}g_1(z)z_3^2+\frac{g_2(z)}{|z|^2}z_1^2z_3^2+\frac{g_3(z)}{|z|^4}z_1^4z_3^2-\frac{g_3(z)}{|z|^4}z_1^2z_2^2z_3^2\,dz\\
=& 2\left(G^1_{1,0,0}+G^2_{1,1,0}+G^3_{2,1,0}-G^3_{1,1,1}\right)
\end{split}
\end{equation}
Next we trial $\nabla Q(x)=(e_1\otimes e_1-e_2\otimes e_2)\otimes e_2$. In this case $Q_{\alpha\beta,\gamma}Q_{\alpha\gamma,\beta}=Q_{\alpha\beta,\beta}Q_{\alpha\gamma,\gamma}=1$, $Q_{\alpha\beta,\gamma}Q_{\alpha\beta,\gamma}=2$. Therefore
\begin{equation}
\begin{split}
\frac{1}{(s^*)^2}K_1=&2L_1+L_2+L_3\\
=& L\nabla Q \cdot \nabla Q\\
=& \int_{\r3}K(z)_{1111}z_{2}^2-2K(z)_{1122}z_2^2+K_{2222}z_2^2\,dz\\
=& \int_{\r3}g_1(z)z_2^2+\frac{g_2(z)}{|z|^2}z_1^2z_2^2+\frac{g_3(z)}{|z|^4}z_1^4z_2^2\\
&-2\frac{g_3(z)}{|z|^4}z_1^2z_2^4+g_1(z)z_2^2+\frac{g_2(z)}{|z|^2}z_2^4+\frac{g_3(z)}{|z|^4}z_2^6\,dz. \\
=& 2G^1_{1,0,0}+G^2_{1,1,0}+G^2_{2,0,0}+G^3_{3,0,0}-G^3_{2,1,0}.
\end{split}
\end{equation}
\end{proof}

\begin{remark}\label{remarkNumericalConstants}
For the sake of illustration, we now consider the simple case where $g_i(z)=c_ig(z)$, and 
\begin{equation}
g(z)=\left\{\begin{array}{l l} 1/|z|^6 \hspace{1cm}& |z|>0.1\\ 0 & \text{ else}\end{array}\right.
\end{equation} for constants $c_i$ to illustrate. This is consistent with the quadratic part of the London dispersion forces formula for uniaxial molecules \cite{london1930theorie}, with a cutoff region to avoid singularities as employed in the derivation of Maier-Saupe. It should be noted however that while the London dispersion forces can be expressed in the form of \eqref{eqLondonDispersion} with $g_i=c_ig$ as here, the exact constants $c_i$ unfortunately give rise to a bilinear form lacking the necessary coercivity for our analysis (explicitly, $c_1=\frac{C}{9}$, $c_2=-\frac{2C}{3}$, $c_3=C$ for a material constant $C>0$). Firstly, from the relationship in \Cref{corollaryk0}, we numerically find that 
\begin{equation}
k_0\approx 4058c_1+1353c_2+811c_3.
\end{equation}
Then by \Cref{propFrankConstants} we find that $G$ constants first, which evaluate as 
\begin{equation}
\begin{array}{r l r l r l}
G^1_{1,0,0}\approx & 41.32c_1, \hspace{1cm}& G^2_{1,1,0}\approx & 8.264c_2,\hspace{1cm} &G^2_{2,0,0}\approx & 24.79c_2\\
G^3_{1,1,1}\approx & 1.181c_3, &G^3_{2,1,0}\approx & 3.542c_3, &G^3_{3,0,0}\approx & 17.71c_3
\end{array}
\end{equation}
This gives our Frank constants as
\begin{equation}
\begin{split}
\frac{1}{(s^*)^2}K_1\approx &83c_1+33c_2+14c_3,\\
\frac{1}{(s^*)^2}K_2\approx &83c_1+17c_2+4.7c_3,\\
K_3=&K_1.
\end{split}
\end{equation}
In particular, if $c_i$ are all roughly equal, so $c_i \approx c$ for $i=2,3$, we can quantify the accuracy of the one-constant approximation, $K_1=K_2=K_3$. The relative accuracy of this approximation is $\frac{|K_1-K_2|}{K_1}\approx\frac{17c_2+9.4c_3}{83c_1+33c_2+14c_3}\approx\frac{(17+9.4)c}{(83+33+14)c} \approx 0.2$. Alternatively if the isotropic constant $c_1$ dominates, so $c_1\geq c_i\geq 0$ for $i=2,3$, then $\frac{|K_1-K_2|}{K_1}\approx\frac{17c_2+9.4c_3}{83c_1+33c_2+14c_3}\leq \frac{(17+9.4)c_1}{83c_1}\approx 0.3$. 
\end{remark}

\section*{Acknowledgements}
The author would like to thank John M. Ball, Epifanio G. Virga, Claudio Zannoni and Giacomo Canevari for insightful discussions that have benefited this work. The research leading to these results has received funding from the European Research Council under the European Union's Seventh Framework Programme (FP7/2007-2013) / ERC grant agreement n$^{\circ}$ 291053.

\bibliographystyle{acm}
\bibliography{bibl}

\end{document}